\documentclass[12pt,a4paper,reqno]{amsart} 
\usepackage{amsmath,amsfonts,amssymb,mathtools}
\usepackage[left=2.9cm,right=2.9cm,top=2.66cm,bottom=2.66cm]{geometry}

\usepackage[english]{babel}
\usepackage{mathrsfs}
\usepackage{enumitem}
\usepackage{slashed}
\usepackage{microtype}

\numberwithin{equation}{section}
\allowdisplaybreaks[2]
\newtheorem{theorem}{Theorem}[section]
\newtheorem{proposition}[theorem]{Proposition}
\newtheorem{lemma}[theorem]{Lemma}
\newtheorem{corollary}[theorem]{Corollary}

\usepackage[sorted-cites]{amsrefs} 
\usepackage[colorlinks=true, allcolors=blue]{hyperref} 

\newtheorem{mthm}{Theorem}

\theoremstyle{definition}
\newtheorem{example}[theorem]{Example}
\newtheorem{definition}[theorem]{Definition}
\newtheorem{remark}[theorem]{Remark}

\newcommand{\sph}{\mathbf{S}}    
\newcommand{\sphm}{\sph^m} 
\newcommand{\Ric}{\operatorname{Ric}}  
\newcommand{\scal}{\operatorname{scal}} 
\newcommand{\scalM}{\operatorname{scal}_{g_M}} 
\newcommand{\supp}{\operatorname{supp}} 
\newcommand{\dist}{\operatorname{dist}} 

\newcommand{\Af}{\widehat{\mathbf{A}}} 
\newcommand{\ch}{\operatorname{ch}} 
\newcommand{\ind}{\operatorname{ind}} 
\newcommand{\sflow}{\operatorname{sf}} 
\newcommand{\indrel}{\operatorname{ind-rel}} 
\newcommand{\sfrel}{\operatorname{sf-rel}} 
\newcommand{\dbl}{\textrm{D}} 

\makeatletter
\@namedef{subjclassname@2020}{\textup{2020} Mathematics Subject Classification}
\makeatother

\begin{document}

\title[Sharp quantitative Llarull's theorem and rigidity]{Sharp quantitative Llarull's theorem and scalar-curvature rigidity with boundary}

\author[Daoqiang Liu]{Daoqiang Liu}

\address{Chern Institute of Mathematics \& LPMC, Nankai University, Tianjin 300071, China}
\email{\href{mailto:dqliu@nankai.edu.cn}{dqliu@nankai.edu.cn}}
\urladdr{\href{https://www.dqliu.cn}{www.dqliu.cn}}

\subjclass[2020]{Primary 53C21, 53C27, Secondary 53C23, 58J30, 58J32}

\keywords{Callias operators, noncompact manifolds, scalar curvature, spectral flow}

\begin{abstract}
We prove a sharp quantitative extension of Llarull's theorem to complete noncompact spin manifolds. Let $\Phi:M\to\sphm$ be an area-decreasing map of nonzero degree that is locally constant at infinity, and assume that $\scalM\geq m(m-1)$ on $\supp(d\Phi)$. If $\lambda_1(M,g_M)=0$, then $\inf_M\scalM<0$; if $\lambda_1(M,g_M)>0$, then
\[
    \inf_M\scalM\leq -\frac{4m}{m-1}\lambda_1(M,g_M).
\]
The coefficient is sharp in every dimension, as shown by an explicit family of warped metrics.
We also prove a boundary analogue under the same support-localized scalar-curvature hypothesis, without requiring the strict scalar-curvature positivity condition near the boundary used in Liu--Liu's doubling argument.
Applications include a Geroch-type obstruction for connected sums with compact area-enlargeable manifolds and a Ricci-flat splitting theorem for complete metrics with nonnegative scalar curvature and mean-convex boundary.
\end{abstract}

\maketitle

\section{Introduction}\label{sec:intro}
Dirac operators have been central to the study of positive scalar curvature on spin manifolds since Lichnerowicz's vanishing theorem~\cite{Lic63}; see also Gromov--Lawson~\cite{GL83}. A model rigidity result is Llarull's theorem~\cite{Ll98}: if a closed $m$-dimensional spin manifold $(M,g_M)$ satisfies $\scal_{g_M}\geq m(m-1)$ and admits an area-decreasing map of nonzero degree to the unit sphere $\sphm$ (the standard unit $m$-sphere with its canonical metric $g_{\sphm}$), then the map is an isometry for $m\geq3$.

The noncompact analogue was established qualitatively by Zhang~\cite{Zh20} and Li--Su--Wang--Zhang~\cite{LSWZ24+}, using deformed Dirac operators and spectral flow, respectively, in response to a question of Gromov (see an earlier version of~\cite{Gro23}). Related Llarull-type rigidity and spectral-flow developments include~\cite{GS02},\cite{Lot21},\cite{Su19},\cite{SWZ22},\cite{BZ25+}. We record the noncompact result in the form needed here.
We use the following formulation.
\begin{theorem}[{\cite{Zh20},\cite{LSWZ24+}}]\label{thm:complete_noncompact_Llarull}
Let $(M,g_M)$ be an $m$-dimensional ($m\geq 2$) complete connected noncompact Riemannian spin manifold. Let $\Phi: M\to \sphm$ be a smooth area-decreasing map which is locally constant at infinity and of nonzero degree. Assume that $\scalM \geq m(m-1)$ on $\supp(d\Phi)$. Then
\[
\inf_M \scalM <0.
\]
\end{theorem}

Recall that a smooth map $\Phi: M\to \sphm$ is called \textit{area-decreasing} if the area contraction constant of $\Phi$ at each point $p\in M$ satisfies
\begin{equation}\label{defn:contraction_constant}
    \mathbf{a}(p)
    :=\sup_{\eta \in \Lambda^2 T_p M\setminus \{0\} }
      \frac{| \Lambda^2 d_p \Phi(\eta)|_{g_{\sphm}}}{|\eta|_{g_M}} \leq 1.
\end{equation}
where $\Lambda^2d_p \Phi: \Lambda^2 T_p M\to \Lambda^2  T_{\Phi(p)}\sphm$ is the induced map of $\Phi$ on $2$-vectors at $p$.
We write
\[
\supp(d\Phi):=\overline{\{p\in M\colon d_p\Phi\neq 0\}}.
\]
For a map that is locally constant outside a compact subset, we use the compactly supported degree convention. Since every point outside a compact set is critical, every regular value has compact preimage; its degree is the signed count of those preimages. The count is independent of the chosen regular value by the usual cobordism argument. In the boundary case we use the analogous relative convention for maps that are locally constant near $\partial M$.
This theorem gives the basic obstruction to nonnegative scalar curvature in the presence of a nonzero-degree area-decreasing map. The quantitative question is whether the amount of scalar-curvature negativity can be controlled by an intrinsic spectral invariant of $(M,g_M)$.

Shi's Theorem~1.9 in~\cite{Shi24+} is a localized, distance-sensitive estimate. In addition to the curvature bound on $\supp(d\Phi)$, it assumes that, for some $\sigma>0$, there is a compact smooth domain $K\subset M$ with $\supp(d\Phi)\subset K^{\circ}$ and
\[
    \scalM\geq \sigma^2m(m-1)\quad\text{on }K\setminus\supp(d\Phi),
    \qquad
    0<\delta:=\dist_{g_M}(\supp(d\Phi),\partial K)<\frac{\pi}{\sigma m}.
\]
Its conclusion is expressed in terms of this distance $\delta$. Earlier compact inequalities of Ono~\cite{Ono88} and Davaux~\cite{Dav03} use spectral data in different settings. Our main estimate uses the bottom spectrum $\lambda_1(M,g_M)$, assumes the scalar-curvature lower bound only on $\supp(d\Phi)$, and has no such surrounding-domain hypothesis.

Our contributions are as follows. We prove the sharp estimate
\[
\inf_M\scalM\leq-\frac{4m}{m-1}\lambda_1(M,g_M),
\]
with the strict conclusion $\inf_M\scalM<0$ when $\lambda_1(M,g_M)=0$. The estimate retains the bottom spectrum in the noncompact setting and has an optimal coefficient. Its hypothesis is localized to $\supp(d\Phi)$, in contrast with Shi's distance-sensitive estimate, which also assumes a scalar-curvature lower bound on the complement of the support inside a compact smooth domain. We construct equality examples in every dimension. We also prove a boundary theorem under the same support-localized hypothesis; unlike the doubling argument of Liu--Liu~\cite{LL26+}, it does not require strict scalar-curvature positivity near the boundary. Finally, we obtain a boundary Geroch-type obstruction for connected sums with compact area-enlargeable manifolds and a Ricci-flat splitting theorem, using the deformation result of Cruz--Santos.
The companion preprint \cite{Liu26b} treats the relative $\widehat A$-cowaist and shares the numerical values of the warped model. The additional feature of the present sharpness construction is a central positive-scalar-curvature region containing $\supp(d\Phi)$, as required by the Llarull-map hypothesis.

The proofs combine relative Dirac bundles with the Callias index and spectral-flow formalisms. A regularization argument for $L^2$ kernel sections allows the compactly supported integral estimates to be used on the complete noncompact manifolds and in the boundary setting.

Since $4m/(m-1)>1$, the theorem in particular implies the strict estimate
\[
    \inf_M \scalM< -\lambda_1(M,g_M).
\]
The sharpness construction below shows that the coefficient $4m/(m-1)$ cannot be improved in the class considered here.
Recall that the bottom spectrum $\lambda_1(M,g_M)$ of the Laplacian on a complete manifold $(M,g_M)$ is characterized by
\begin{equation}\label{defn:bottom_spectrum}
    \lambda_1(M,g_M)=\inf_{u \in C_c^{\infty}(M)\setminus \{0\} } \dfrac{\int_M |d u|^2_{g_M} dV_{g_M}}{\int_M u^2 dV_{g_M}},
\end{equation}
where $C_c^{\infty}(M)$ is the space of smooth functions on $M$ with compact support.


With this notation, our first main theorem is the following quantitative strengthening of Theorem~\ref{thm:complete_noncompact_Llarull}.

\begin{mthm}\label{mthm:quantitative_Llarull_theorem}
    Let $(M,g_M)$ be an $m$-dimensional ($m\geq 2$) complete connected noncompact Riemannian spin manifold. 
    Let $\Phi: M\to \sphm$ be a smooth area-decreasing map which is locally constant at infinity and of nonzero degree. Assume that $\scalM \geq m(m-1)$ on $\supp(d\Phi)$. 
    \begin{enumerate}[start=1,label=\textup{(\arabic*)}]
    \item  
    If $\lambda_1(M,g_M)=0$, then
    \[
        \inf_M \scalM <0.
    \] 
    \item
    If $\lambda_1(M,g_M)>0$, then
    \[
        \inf_M \scalM \leq - \frac{4m}{m-1} \lambda_1(M,g_M).
    \]
    \end{enumerate}
\end{mthm}

The coefficient in Theorem~\ref{mthm:quantitative_Llarull_theorem} is optimal in the full class of manifolds and maps appearing there.
\begin{proposition}[Sharpness of the quantitative estimate]\label{prop:sharpness_quantitative_Llarull}
For every integer $m\geq 2$, there exists a complete connected noncompact spin manifold $(M,g_M)$ and a smooth area-decreasing map $\Phi:M\to \sphm$, locally constant outside a compact set and of degree one, such that
\[
\begin{aligned}
    \scal_{g_M}&\geq m(m-1)\quad\text{on }\supp(d\Phi),\\
    \inf_M\scal_{g_M}&=-m(m-1),\\
    \lambda_1(M,g_M)&=\frac{(m-1)^2}{4}.
\end{aligned}
\]
Consequently,
\[
    \inf_M\scal_{g_M}=-\frac{4m}{m-1}\lambda_1(M,g_M).
\]
\end{proposition}
The construction and proof are given in Section~\ref{sec:sharpness_quantitative_Llarull}.

As an immediate consequence of Theorem~\ref{mthm:quantitative_Llarull_theorem}, we obtain the bottom spectrum estimate under a scalar curvature lower bound.
\begin{corollary}
    Let $(M,g_M)$ be an $m$-dimensional ($m\geq 2$) complete connected noncompact Riemannian spin manifold such that $\scalM\geq \kappa$ for some constant $\kappa < 0$. 
    Let $\Phi: M\to \sphm$ be a smooth area-decreasing map which is locally constant at infinity and of nonzero degree. Assume that $\scalM \geq m(m-1)$ on $\supp(d\Phi)$. Then
    \[
        \lambda_1(M,g_M) \leq - \frac{m-1}{4m} \kappa.
    \]
\end{corollary}

By adapting the arguments employed in Theorem~\ref{mthm:quantitative_Llarull_theorem}, we also establish a boundary version of Theorem~\ref{mthm:quantitative_Llarull_theorem}. This drops the strict positivity condition on scalar curvature near the boundary, which was imposed in Liu--Liu~\cite{LL26+} via a doubling trick.

\begin{mthm}\label{mthm:quantitative_Llarull_theorem_boundary}
    Let $(M,g_M)$ be an $m$-dimensional ($m\geq 2$) complete connected noncompact Riemannian spin manifold with compact mean-convex boundary. 
    Let $\Phi: M\to \sphm$ be a smooth area-decreasing map which is locally constant at infinity and near the boundary and of nonzero degree. Assume that $\scalM \geq m(m-1)$ on $\supp(d\Phi)$. Then
    \[
          \inf_M \scalM < 0.
    \]
\end{mthm}

Chodosh--Li~\cite{CL24} proved that for any manifold $M$ of dimension $m\leq 7$, the connected sum $\mathbf{T}^m \# M$ cannot admit a complete metric of positive scalar curvature. Wang--Zhang~\cite{WZ22} and Su~\cite{Su25+} proved that for any spin manifold $M$, $\mathbf{T}^m \# M$ cannot admit a complete metric of positive scalar curvature. A similar question arises in the context of manifolds with boundary.
Recall that an $m$-dimensional ($m\geq2$) oriented compact connected Riemannian manifold $M$ with boundary is area enlargeable if for any $\varepsilon>0$ there exists a finite covering $\widehat{M}\to M$ and an area $\varepsilon$-decreasing smooth map $f:\widehat{M}\to \sphm$ of nonzero degree which is locally constant near boundary, where ``area $\varepsilon$-decreasing'' means that the area contraction constant of $f$ at each point $p\in \widehat{M}$ satisfies $\mathbf{a}(p)\leq \varepsilon$ (cf.~\cite{BH23}).
A key application of Theorem~\ref{mthm:quantitative_Llarull_theorem_boundary} is the following corollary.
\begin{corollary}\label{cor:geroch_conjecture_boundary}
Let $W$ be a compact connected spin area enlargeable manifold with boundary and $M$ a connected noncompact spin manifold without boundary of equal dimension.
Then the connected sum $\mathcal{M}=W\# M$ (with the connected sum performed in the interior) does not admit a complete metric with positive scalar curvature and mean-convex boundary.
\end{corollary}

This corollary is closely related to the following scalar-mean curvature splitting theorem.
\begin{mthm}\label{mthm:rigidity}
Let $W$ be an $m$-dimensional ($m\geq 3$) compact connected spin area enlargeable manifold with boundary and $M$ a connected noncompact spin manifold without boundary of equal dimension.
Suppose that the connected sum $\mathcal{M}=W\# M$ has a complete metric $g_{\mathcal{M}}$ of nonnegative scalar curvature and mean-convex boundary. Then $g_{\mathcal{M}}$ is Ricci-flat, the boundary $\partial\mathcal{M}$ is connected, and $\mathcal{M}$ is isometric to a product
\[
    (\partial\mathcal{M}\times[0,\infty),\,g_{\partial\mathcal{M}}+dt^2),
\]
where $(\partial\mathcal{M},g_{\partial\mathcal{M}})$ is Ricci-flat.
\end{mthm}

Note that $\mathbf{T}^{m-1}\times [0,1]$ is spin and area enlargeable for $m\geq 2$. When $M$ is noncompact, $\mathcal{M}=(\mathbf{T}^{m-1}\times[0,1])\#M$ has two boundary components diffeomorphic to $\mathbf{T}^{m-1}$. Theorem~\ref{mthm:rigidity} rules out a complete metric with nonnegative scalar curvature and mean-convex boundary, since it would force $\partial\mathcal{M}$ to be connected.
\begin{corollary}\label{cor:torical_band_noncompact}
Let $M$ be a connected noncompact $m$-dimensional spin manifold without boundary, where $m\geq 3$.
Then $\mathcal{M}=(\mathbf{T}^{m-1}\times [0,1])\# M$ does not admit a complete metric of nonnegative scalar curvature and mean-convex boundary.
\end{corollary}

\begin{remark}
This corollary answers the question raised by Cruz--Santos in the discussion following~\cite{CS25+}*{Corollary~1.2}, in the spin case.
\end{remark}
\begin{remark}
For the case of a compact manifold $M$ (possibly non-spin, $3\leq m\leq 7$), the corresponding obstruction was proved by Gromov--Lawson~\cite{GL83} (also Chai~\cite{Chai18+}*{Theorem~3}). The relation with positive mass theorem for asymptotically flat manifolds with a noncompact boundary can be found in~\cite{Chai18+}*{Section~4.4}.
\end{remark}

\textbf{Notation.}
Unless otherwise specified, throughout this article all manifolds are smooth, connected, oriented, and have dimension at least two.
Additionally, the notion of completeness refers to metric completeness when $M$ has nonempty boundary.
A boundary is called \textit{mean-convex} if its mean curvature with respect to the outward unit normal is nonnegative.

\textbf{Plan of the paper.}
In Section~\ref{sec:Callias_operator_and_spectral_flow}, we introduce the notions and some necessary technical tools for deformed Dirac operators and spectral flow. 
Section~\ref{sec:formula_in_codimension_zero} is devoted to the proof of Theorem~\ref{mthm:quantitative_Llarull_theorem}.
Section~\ref{sec:sharpness_quantitative_Llarull} contains an illustrative example and the sharpness construction.
In Section~\ref{sec:proof_boundary}, we prove Theorem~\ref{mthm:quantitative_Llarull_theorem_boundary}, Corollary~\ref{cor:geroch_conjecture_boundary} and Theorem~\ref{mthm:rigidity}.


\section{Preliminaries}\label{sec:Callias_operator_and_spectral_flow}

This section reviews the fundamental properties of deformed Dirac operators (also called Callias operators).
Let $(M,g_M)$ be an $m$-dimensional complete Riemannian manifold, possibly noncompact and with compact boundary. Write $M^{\circ}$ for its interior.

\subsection{Gromov-Lawson pairs and Callias operators}
Following~\cite{GL83} (also~\cite{CZ24}), a \textit{Gromov--Lawson pair} (GL-pair) on $M$ consists of two Hermitian vector bundles $(\mathcal{E},\mathcal{F})$ equipped with metric connections, together with a parallel unitary isomorphism
\[
\mathcal{I} \colon \mathcal{E}|_{M \setminus K} \longrightarrow \mathcal{F}|_{M \setminus K},
\]
which extends to a smooth bundle map on a neighborhood of $\overline{M\setminus K}$,
for a compact subset $K \subset M^{\circ}$. The set $K$ is called the \textit{support} of the pair.

When $M$ is spin, a GL-pair determines a relative Dirac bundle (cf.~\cite{CZ24}*{Example~2.5}). We use the following block convention:
\[
\begin{aligned}
S&:=(\slashed{S}_M\otimes\mathcal{E})\oplus(\slashed{S}_M\otimes\mathcal{F}),\\
c(\xi)&=\begin{pmatrix}c_M(\xi)&0\\0&-c_M(\xi)\end{pmatrix},\\
\mathcal{D}&=\begin{pmatrix}\slashed D_{\mathcal E}&0\\0&-\slashed D_{\mathcal F}\end{pmatrix}.
\end{aligned}
\]
This block convention is an alternative realization of a relative Dirac bundle in the sense of Cecchini--Zeidler: it satisfies the same Clifford and grading relations and yields the same twisting-curvature endomorphism.
Here $c_M$ is the usual spin Clifford multiplication, with $c_M(\xi)^2=-|\xi|^2$. If $m$ is even, let $\gamma_M$ denote the spinor grading. The grading on $S$ is $\operatorname{diag}(\gamma_M,-\gamma_M)$, so
\[
\begin{aligned}
S^+&=(\slashed S_M^+\otimes\mathcal E)\oplus(\slashed S_M^-\otimes\mathcal F),\\
S^-&=(\slashed S_M^-\otimes\mathcal E)\oplus(\slashed S_M^+\otimes\mathcal F).
\end{aligned}
\]
In odd dimensions $S$ is used without a grading. Outside $K$ define the self-adjoint, parallel involution
\[
\begin{aligned}
\sigma:=\operatorname{id}_{\slashed{S}_M}\otimes
\begin{pmatrix}
0 & \mathcal{I}^{*} \\
\mathcal{I} & 0
\end{pmatrix} \colon S|_{M\setminus K}\longrightarrow S|_{M\setminus K}.
\end{aligned}
\]
It anti-commutes with $c(\xi)$ in either parity and is odd with respect to the displayed grading when $m$ is even. The ordinary spectral flow used below in odd dimensions is taken on the ungraded bundle.
A Lipschitz function $f:M\to\mathbf{R}$ is called an \textit{admissible potential} if $f=0$ on the compact support $K$ and there is a compact set $L$ with $K\subset L\subset M$ such that $f$ is equal to a nonzero constant on each component of $M\setminus L$; see~\cite{CZ24}*{Definition~3.1}.
For such \(f\), the product \(f\sigma\) extends by zero to a continuous bundle map on all of \(M\).
\begin{definition}
The \textit{Callias operator} on $S$ associated to the above data is given by
\[
 \mathcal{B}_{f} := \mathcal{D} + f \sigma.
\]  
\end{definition}

For the analysis of Callias operators on a manifold $M$ with compact boundary, one must impose appropriate local boundary conditions. The relevant notion is that of a boundary chirality.
\begin{definition}[\cite{CZ24}]
Let $S$ be a relative Dirac bundle and let $s\colon \partial M\to \{\pm 1\}$ be a locally constant function. The \textit{boundary chirality} on $S$ associated to the choice of signs $s$ is the endomorphism
\[
     \chi := s\, c(\nu^*)\sigma \colon \left. S\right|_{\partial M} \to \left. S\right|_{\partial M},
\]
where $\nu^*$ is the dual covector of the outward unit normal $\nu$ to $\partial M$.
\end{definition}

Our normal convention is outward-pointing. Cecchini--Zeidler~\cite{CZ24} and Shi~\cite{Shi24+} formulate the corresponding boundary chirality using the inward-pointing normal. Since $\nu_{\mathrm{in}}=-\nu_{\mathrm{out}}$, our choice $s=-1$ is their choice $s=+1$. With their convention $H_{\mathrm{in}}=(m-1)^{-1}\operatorname{tr}(-\nabla\nu_{\mathrm{in}})$ and ours $H_{\mathrm{out}}=(m-1)^{-1}\operatorname{tr}(\nabla\nu_{\mathrm{out}})$, the mean-curvature values agree: $H_{\mathrm{in}}=H_{\mathrm{out}}$. All applications of their boundary index and spectral-flow results below use these conversions; the mean curvature $H$ in our formulas is always computed with respect to the outward normal.

The map $\chi$ is a self-adjoint involution. It anti-commutes with $c(\nu^*)$ and commutes with $c(w^*)$ for all $w\in T(\partial M)$. In even dimensions it is even with respect to the grading above.
This leads to the following boundary condition.
\begin{definition}[\cite{CZ24}]
A section $u\in C^{\infty}(M,S)$ satisfies the \textit{local boundary condition} if
\begin{equation}\label{eq:local_boundary_condition}
   \chi \left(u|_{\partial M}\right) = u|_{\partial M}. 
\end{equation}
\end{definition}
For a choice of $s:\partial M\to \{\pm 1\}$, denote by \(\mathcal{B}_{f,s}\) the restriction of \(\mathcal{B}_{f}\) to the domain
\[
\begin{aligned}
   \operatorname{dom}(\mathcal{B}_{f,s})
   :=\overline{\{u\in C_c^{\infty}(M,S):
   \chi(u|_{\partial M})=u|_{\partial M}\}}^{
   \,\|u\|_{L^2}+\|\mathcal{B}_fu\|_{L^2}}.
\end{aligned}
\]
Here the closure is the graph closure in $L^2(M,S)$. No identification with an unweighted global Sobolev space is needed. For bounded Lipschitz $f$, the zeroth-order term $f\sigma$ does not change local elliptic regularity. Thus the graph-domain approximation is understood with the usual $H^1$ trace condition on compact subsets, and the local chirality condition is preserved under the trace map. The regularization argument below uses only this local identification and the graph closure.

 \begin{lemma}\label{lem:kernel_regularization}
 Let $\mathcal B_*$ be one of the operators $\mathcal B_f$, $\mathcal B_{f,s}$, $\mathcal B_f(t)$, or $\mathcal B_{f,s}(t)$ considered below. Assume that the curvature endomorphism in the corresponding Bochner formula is bounded below uniformly by a constant, that $f,d f\in L^\infty(M)$, and, in the family case, that these bounds are uniform for $t\in[0,1]$. If $M$ has boundary, assume in addition that $s=-1$ and that $f\geq0$ and the boundary mean curvature $H\geq0$. If $u\in L^2(M,S)$ belongs to the graph domain of $\mathcal B_*$ and satisfies $\mathcal B_*u=0$, then $u\in H^1(M,S)$ and $|u|\in H^1(M)$. All integral inequalities needed in the spectral estimates below, initially proved for compactly supported smooth sections, remain valid for $u$.
 \end{lemma}

 \begin{proof}
 Choose smooth compactly supported exhaustion cut-offs $0\leq\eta_R\leq1$ with $\eta_R=1$ on the boundary collar when $\partial M\neq\varnothing$, $\eta_R\to1$ pointwise, and $\|d\eta_R\|_{L^\infty}\to0$. Put $v_R=\eta_Ru$. The local boundary condition is preserved and
 \[
  r_R:=\mathcal B_*v_R=c(d\eta_R)u,
 \qquad \|r_R\|_{L^2}\longrightarrow0.
 \]
 Local elliptic regularity for Dirac operators with local elliptic boundary conditions gives $v_R\in H^1$ on each compact set. By the definition of the graph closure, we first approximate $v_R$ by smooth compactly supported sections satisfying the same local boundary condition. Since $f\in W^{1,\infty}$, the term $f\sigma$ is a bounded zeroth-order perturbation. Elliptic regularity and continuity of the boundary trace therefore allow the Bochner--Lichnerowicz and Green identities to pass to the limit. The uniform lower bound for the curvature endomorphism, together with the bounds on $f$ and $d f$ and the boundary identity when present, gives
 \begin{equation}\label{eq:kernel_H1_cutoff_estimate}
     \|\nabla v_R\|_{L^2}^2
     \leq C\bigl(\|r_R\|_{L^2}^2+\|v_R\|_{L^2}^2\bigr).
 \end{equation}
 In the boundary case we use $s=-1$ and $f,H\geq0$. After moving the boundary term to the appropriate side it is nonnegative. Thus $\{v_R\}$ is uniformly bounded in $H^1$.

  Since $v_R\to u$ in $L^2$, weak compactness and the distributional characterization of the covariant derivative give $u\in H^1$ and $v_R\rightharpoonup u$ weakly in $H^1$. Kato's inequality gives a uniform $H^1$ bound for $|v_R|$, while $|v_R|\to|u|$ in $L^2$. After passing to a subsequence, $|v_R|\rightharpoonup|u|$ weakly in $H^1$. For the Penrose estimate, the kernel relation for $v_R$ has residual $r_R$. The extra terms in the Cauchy--Schwarz step are bounded by
 \[
 C|r_R|\bigl(|\nabla v_R|+|f|\,|v_R|\bigr)+C|r_R|^2,
 \]
 and their integrals tend to zero by Cauchy--Schwarz, \eqref{eq:kernel_H1_cutoff_estimate}, and $r_R\to0$ in $L^2$. The Green identity has the same commutator residual, while the choice $\eta_R=1$ near the boundary introduces no new boundary term.

  Finally, choose $C$ so that the curvature endomorphism plus $C$ is nonnegative. After passing to a further subsequence with $v_R\to u$ almost everywhere, Fatou's lemma applied to the nonnegative quadratic form $\langle v_R,(\mathscr R+C)v_R\rangle$, together with weak lower semicontinuity of the gradient terms, including $|d|v_R||^2$, gives the required limiting inequality. The bounded zero-order terms converge by the strong $L^2$ convergence $v_R\to u$. Subtracting the harmless $C\|v_R\|_{L^2}^2$ term then yields the original form. On a complete boundaryless manifold, density of $C_c^\infty(M)$ in $H^1(M)$ allows the Rayleigh quotient defining $\lambda_1$ to be applied to $|u|$.
 \end{proof}

When $m$ is even, $S=S^+\oplus S^-$ is $\mathbf{Z}_2$-graded and
$\mathcal{B}_f$ is odd with off-diagonal components $\mathcal{B}_f^{\pm}:C^{\infty}(M,S^{\pm})\to C^{\infty}(M,S^{\mp})$. Explicitly,
\[
\mathcal{B}_f(u_+,u_-)=(\mathcal{B}_f^-u_-,\mathcal{B}_f^+u_+).
\]
Because \(\chi\) is even with respect to the grading, the domain of \(\mathcal{B}_{f,s}\) decomposes into its intersections with \(L^2(M,S^+)\) and \(L^2(M,S^-)\). Denote the corresponding restrictions by
\[
\begin{aligned}
    \mathcal{B}_{f,s}^{\pm}:=\mathcal{B}_{f,s}\big|_{\operatorname{dom}(\mathcal{B}_{f,s})
    \cap L^2(M,S^{\pm})}
    \colon \operatorname{dom}(\mathcal{B}_{f,s})\cap L^2(M,S^{\pm})
    \longrightarrow L^2(M,S^{\mp}).
\end{aligned}
\]
The self-adjoint Fredholm realization in the boundary case is supplied by~\cite{CZ24}*{Theorem~3.4} in even dimensions and by the odd-dimensional Callias framework in~\cite{Shi24+}. In even dimensions its graded \textit{index} is defined as
\[
    \ind(\mathcal{B}_{f,s}):=\dim (\ker(\mathcal{B}_{f,s}^+)) - \dim (\ker(\mathcal{B}_{f,s}^-)).
\]

\subsection{Spectral flow for odd-dimensional manifolds}
Assume that $\dim M$ is odd and that the GL-pair $(\mathcal{E},\mathcal{F})$ is trivial of rank $l$ on each summand, with support $K$, as in the application below. Let $f$ be an admissible potential, so that $f=0$ on $K$, and let
\[
\rho_M=\rho_M^+\oplus\rho_M^-
\]
be a unitary automorphism of $\mathcal{E}\oplus\mathcal{F}$ such that $\rho_M^+$ and $\rho_M^-$ agree under the parallel identification on $M\setminus K$. Assume also that $\rho_M$ is locally constant outside a compact set and parallel near $\partial M$. Then $[\mathcal{D},\rho_M]$ is bounded, and $\rho_M$ commutes with $\sigma$ near $\partial M$, so it preserves the local boundary condition. In fixed trivializations, $\rho_M^\pm$ are smooth $U(l)$-valued maps.
Write $\nabla^0$ for the fixed trivial connection on each summand. For the trivial connections on $\mathcal{E}$ and $\mathcal{F}$, define
\[
\begin{aligned}
\nabla^{\mathcal{E}\oplus\mathcal{F}}(t):=\begin{pmatrix}
\nabla^0+t(\rho_M^+)^{-1}[\nabla^0,\rho_M^+] & 0\\
0&\nabla^0+t(\rho_M^-)^{-1}[\nabla^0,\rho_M^-]
\end{pmatrix},\hspace{2em} t\in[0,1].
\end{aligned}
\]
On the relative Dirac bundle $S$, let $\nabla(t)$ be the connection induced by this family and let $\mathcal{D}(t)$ be its Dirac operator. Then
\[
\mathcal{D}(t)=\mathcal{D}+t\rho_M^{-1}[\mathcal{D},\rho_M]=(1-t)\mathcal{D}+t\rho_M^{-1}\mathcal{D}\rho_M.
\]
Because $f=0$ on $K$, while $\rho_M$ agrees across the two summands on $M\setminus K$, one has $\rho_M^{-1}(f\sigma)\rho_M=f\sigma$. Hence $\rho_M^{-1}\mathcal{B}_f\rho_M=\mathcal{D}(1)+f\sigma$ and the path
\[
\mathcal{B}_f(t):=(1-t)\mathcal{B}_f+t\rho_M^{-1}\mathcal{B}_f\rho_M=\mathcal{D}(t)+f\sigma,\quad t\in[0,1],
\]
preserves the local boundary condition. The corresponding self-adjoint Fredholm operators $\mathcal{B}_{f,s}(t)$ form a continuous path in the Riesz topology (cf.~\cite{BMJ05},~\cite{Lesch05}).

\begin{definition}[{\cite{Shi24+}}]
The \textit{spectral flow} of the family \(\{\mathcal{B}_{f,s}(t)\}_{t\in[0,1]}\), denoted by $\sflow(\mathcal{B}_{f,s},\rho_M)$, is the number of crossings from negative to nonnegative minus the number of crossings in the opposite direction, counted with multiplicity. We use the endpoint convention of~\cite{Shi24+}.
\end{definition}

\begin{remark}
    If every single operator $\mathcal{B}_{f,s}(t)$ in such a family is invertible, there cannot be any eigenvalue changing its sign when $t$ varies from $0$ to $1$, and therefore, the spectral flow $\sflow(\mathcal{B}_{f,s},\rho_M)$ is forced to vanish.
\end{remark}

\subsection{Relative index and relative spectral flow}
Let $(\mathcal{E},\mathcal{F})$ be a GL-pair on $M$ with support $K$.
Choose a smooth compact spin submanifold $\mathcal{W}$ containing $\partial M$ with boundary, whose interior contains $K$. Denote by $\mathcal{W}^-$ a copy of $\mathcal{W}$ with the opposite orientation. The double $\dbl\mathcal{W}:=\mathcal{W}\cup_{\partial\mathcal{W}}\mathcal{W}^-$ carries the natural spin structure induced from $\mathcal{W}$. Choose a product metric on a collar of $\partial\mathcal{W}$ and extend it smoothly to a metric on $\dbl\mathcal{W}$. Using the parallel identification $\mathcal{I}$ on the collar, glue $\mathcal{E}$ over $\mathcal{W}$ to $\mathcal{F}$ over $\mathcal{W}^-$. Denote the resulting bundle on $\dbl\mathcal{W}$ by $V(\mathcal{E},\mathcal{F})$. It agrees with $\mathcal{E}$ over $\mathcal{W}$ and with $\mathcal{F}$ over $\mathcal{W}^-$. The resulting index is independent of the auxiliary product metric by homotopy invariance.
If \(\dim M\) is even, the \textit{relative index} of \((\mathcal{E},\mathcal{F})\) is the index of the spin Dirac operator on \(\dbl\mathcal{W}\) twisted by $V(\mathcal{E},\mathcal{F})$ (see~\cite{CZ24}):
\begin{equation}\label{eq:relative_index}
\begin{aligned}
    \indrel(M, \mathcal{E},\mathcal{F})
    := \ind\bigl(\slashed{D}_{\dbl\mathcal{W},V(\mathcal{E},\mathcal{F})}\bigr).
\end{aligned}
\end{equation}

If $\dim M$ is odd and $\mathcal{E},\mathcal{F}$ are trivial rank-$l$ bundles over $\mathcal{W}$, let $\rho^+,\rho^-:\mathcal{W}\to U(l)$ be smooth unitary maps. Assume they agree on $\mathcal{W}\setminus K$ and are locally constant near $\partial\mathcal{W}$. They glue to a smooth unitary map $\widetilde{\rho}:\dbl\mathcal{W}\to U(l)$. In the constructions below, their common collar values extend as a locally constant map over each component of $M\setminus\mathcal{W}$. Extend $\rho^+$ and $\rho^-$ separately using these values, and write $\rho_M=\rho_M^+\oplus\rho_M^-$ for the resulting block automorphism on $\mathcal E\oplus\mathcal F$. Its action on $S$ includes the identity on the spinor factor. Thus $\widetilde\rho$ is the glued rank-$l$ action on the double, whereas $\rho_M$ is the two-block action on $M$.

\begin{definition}[Relative spectral flow]
The \textit{relative spectral flow} of the triple $(\mathcal{E},\mathcal{F},\rho_M)$ is the spectral flow of the corresponding family of spin Dirac operators on $\dbl\mathcal{W}$, twisted by the relative bundle $V(\mathcal{E},\mathcal{F})$; explicitly,
\begin{equation}\label{eq:relative_spectral_flow}
\begin{aligned}
    \sfrel(M, \mathcal{E},\mathcal{F}, \rho_M)
    := \sflow\bigl(
    \slashed{D}_{\dbl\mathcal{W},V(\mathcal{E},\mathcal{F})},
    \widetilde{\rho}\bigr).
\end{aligned}
\end{equation}
\end{definition}

\subsection{Spectral estimates}
In this subsection, we assume that \(\dim M\) is odd.
In the quadratic and integral identities below, brackets denote the real part of the Hermitian inner product. The same convention is used in the even-dimensional estimates.
Let \(\{e_i\}_{i=1}^m\) be a local orthonormal frame on \(TM\) with dual coframe \(\{e^i\}_{i=1}^m\).
For the Dirac operator \(\mathcal{D}(t)\) the Bochner--Lichnerowicz--Schr\"{o}dinger--Weitzenb\"ock formula (cf.~\cite{LM89}) reads
\begin{equation}\label{eq:BLW}
    \mathcal{D}^2(t) = \nabla^*\nabla(t) + \mathscr{R}(t),
\end{equation}
where
\[
\nabla^*\nabla(t) := -\sum_i\bigl(\nabla_{e_i}(t)\nabla_{e_i}(t)-\nabla_{\nabla^{TM}_{e_i}e_i}(t)\bigr)
\]
is the connection Laplacian and
\begin{equation}\label{eq:curvature_endomorphism}
    \mathscr{R}(t) := \sum_{i<j} c(e^i)c(e^j)\,R^{\nabla(t)}_{e_i,e_j} 
\end{equation}
is the curvature endomorphism (\(R^{\nabla(t)}\) denotes the curvature tensor of \(\nabla(t)\)).

Recall that the bilinear Green formula~\cite{Taylor11}*{Proposition~9.1} yields, for $u,v\in C_c^{\infty}(M,S)$,
\begin{equation}\label{eq:Green_formula_Dirac}
\begin{aligned}
    \int_M \langle \mathcal{D}(t) u, v\rangle\,dV
    =\int_M \langle u, \mathcal{D}(t) v\rangle\,dV - \int_{\partial M} \langle u, c(\nu^*) v\rangle\,dA.
\end{aligned}
\end{equation}
Moreover, for $u\in C_c^{\infty}(M,S)$,
\begin{equation}\label{eq:Green_formula_connection}
\begin{aligned}
    \int_M |\nabla(t) u|^2\,dV
    = \int_M \langle u,\nabla^*\nabla(t) u \rangle\,dV + \int_{\partial M} \langle u, \nabla_{\nu}(t) u \rangle\,dA.
\end{aligned}
\end{equation}

Taking $v=\mathcal{D}(t)u$ in \eqref{eq:Green_formula_Dirac}, and using \eqref{eq:BLW} and \eqref{eq:Green_formula_connection}, we obtain for $u\in C_c^{\infty}(M,S)$,
\begin{equation}\label{eq:BLW+Green}
\begin{aligned}
    \int_M |\mathcal{D}(t)u|^2 dV=& \int_M|\nabla(t) u|^2 dV + \int_M \langle u, \mathscr{R}(t) u\rangle dV \\
      & - \int_{\partial M} \langle u, c(\nu^*) \mathcal{D}(t) u + \nabla_{\nu}(t) u \rangle dA.
\end{aligned}
\end{equation}

Since
\[
\begin{aligned}
|\mathcal{B}_f(t) u|^2
&= |\mathcal{D}(t) u|^2
 + \langle\mathcal{D}(t)u, f\sigma u\rangle
 + \langle f\sigma u, \mathcal{D}(t) u \rangle\\
&\quad + f^2|u|^2.
\end{aligned}
\]
we have
\begin{equation}\label{eq:integral_Callias_square}
\begin{aligned}
\int_M |\mathcal{B}_f(t) u|^2 dV 
= &  \int_M |\mathcal{D}(t) u|^2\,dV\\
& + \int_M \left\langle u,
   \underbrace{(\mathcal{D}(t) f\sigma+f\sigma \mathcal{D}(t))}_{=c(d f)\sigma}
   u\right\rangle\,dV\\
& + \int_M f^2|u|^2\,dV
 - \int_{\partial M} \langle u, c(\nu^*) f\sigma u \rangle\,dA. 
\end{aligned}
\end{equation}

On the restricted bundle \(S^{\partial}:=S|_{\partial M}\), we define a Clifford multiplication and a connection by
\[
\begin{aligned}
   c^{\partial}(e^i)&:=-c(e^i)c(\nu^*),\\
   \nabla^{\partial}_{e_i}(t)&:=\nabla_{e_i}(t)
   -\frac{1}{2}c^{\partial}(\nabla_{e_i}(t)\nu^*).
\end{aligned}
\]
where $\{e_i\}_{i=1}^{m-1}$ is a local orthonormal frame on $T(\partial M)$ with dual coframe \(\{e^i\}_{i=1}^{m-1}\), and $\nu^*$ is the dual covector of the outward unit normal $\nu=e_m$ to $\partial M$.
Then
\[
\mathcal{A}(t):=\sum_{i=1}^{m-1} c^{\partial}(e^i)\nabla^{\partial}_{e_i}(t)
\]
is the associated family of boundary Dirac operators. One checks that
\[
\begin{aligned}
\mathcal{A}(t)c(\nu^{*})&=-c(\nu^{*})\mathcal{A}(t),\\
\mathcal{A}(t)\sigma&=\sigma\mathcal{A}(t),\\
\chi\mathcal{A}(t)&=-\mathcal{A}(t)\chi.
\end{aligned}
\]
Consequently, for any \(u\in C^{\infty}(M,S)\) satisfying the local boundary condition \eqref{eq:local_boundary_condition},
\begin{equation}\label{eq:vanishing_boundary_term}
    \langle u|_{\partial M}, \mathcal{A}(t) u|_{\partial M} \rangle = 0.
\end{equation}

Let \(H:=\frac{1}{m-1}\sum_{i=1}^{m-1}\langle e_i,\nabla_{e_i}\nu\rangle\) be the mean curvature of \(\partial M\) with respect to $\nu$. 
A direct computation using the definitions yields the fundamental boundary identity (cf.~\cite{Bar96}*{Proposition~2.2}):
\begin{equation}\label{eq:boundary_identity}
  \mathcal{A}(t)=\frac{m-1}{2}\,H + c(\nu^{*})\mathcal{D}(t) + \nabla_{\nu}(t).  
\end{equation}

For $t\in [0,1]$, we define a family of Penrose operators on $S$ by
\[
\mathcal{P}_{\xi}(t)u:=\nabla_{\xi}(t)u+\frac1m c(\xi^*)\mathcal{D}(t)u,\quad \xi \in TM.
\]
For a local orthonormal frame $\{e_i\}_{i=1}^m$, we use the notation
\[
    |\mathcal{P}(t)u|^2:=\sum_{i=1}^m|\mathcal{P}_{e_i}(t)u|^2;
\]
this quantity is independent of the chosen frame.
It satisfies (cf.~\cite{BHMMM15}*{Section~5.2})
\begin{equation}\label{eq:Penrose_identity}
    |\nabla(t) u|^2 = |\mathcal{P}(t)u|^2 +\frac{1}{m} |\mathcal{D}(t) u|^2.
\end{equation}
Combining \eqref{eq:BLW+Green}, \eqref{eq:integral_Callias_square}, \eqref{eq:vanishing_boundary_term}, \eqref{eq:boundary_identity} and \eqref{eq:Penrose_identity}, we obtain
\begin{equation}\label{eq:Penrose_spectral_estimates}
\begin{aligned}
   \int_M  |\mathcal{B}_{f,s}(t) u|^2 dV
    =& \frac{m}{m-1} \int_M \Big( |\mathcal{P}(t) u|^2 + \langle u, \mathscr{R}(t) u \rangle \Big) dV \\
     & + \int_M  \left \langle u, f^2 u +  c(d f) \sigma u \right \rangle  dV\\
     & - \int_{\partial M} \Big ( sf - \frac{m}{2} H \Big) |u|^2 dA. \\
\end{aligned}
\end{equation}

Assume now that \(f\) is an admissible potential and that $u$ is a nontrivial spinor satisfying \(\mathcal{B}_{f,s}(t)u = 0\) for some \(t \in [0,1]\). We derive a lower bound for the Penrose term.
Following~\cite{HKKZ24}, set \(v_i(t):=c(e^i)\nabla_{e_i}(t)u + \frac{f}{m}\sigma u\).
Observe that \(\sum_{i=1}^m v_i(t) = 0\). Writing \(v(t) = (v_1(t), \bar{v}(t))\) and applying the Cauchy--Schwarz inequality to the relation \(\sum_{i=2}^m v_i(t) = -v_1(t)\), we find
\[
(m-1)|\bar{v}(t)|^2 \geq |v_1(t)|^2.
\]
Since \(|\mathcal{P}(t)u|^2 = \sum_{i=1}^m |v_i(t)|^2 = |v_1(t)|^2 + |\bar{v}(t)|^2\), it follows that
\begin{equation}\label{eq:Penrose_Kato}
|\mathcal{P}(t)u|^2 \geq \frac{m}{m-1}|v_1(t)|^2 = \frac{m}{m-1}\Bigl|\nabla_{e_1}(t)u - \frac{f}{m}c(e^1)\sigma u\Bigr|^2.
\end{equation}

Fix a constant \(c > 0\) such that \(\alpha := \frac{m}{m-1} - \frac{1}{4c} > 0\). From \eqref{eq:Penrose_Kato} we deduce
\[
\begin{aligned}
|\mathcal{P}(t)u|^2
& \geq \frac{m}{m-1}|\nabla_{e_1}(t)u|^2
     - \frac{2}{m-1}f\langle\nabla_{e_1}(t)u,\,c(e^1)\sigma u\rangle\\
&\quad + \frac{1}{m(m-1)}f^2|u|^2.
\end{aligned}
\]
Completing the square with respect to the first two terms yields
\[
\begin{aligned}
|\mathcal{P}(t)u|^2 \geq {}&
\frac{1}{4c}|\nabla_{e_1}(t)u|^2\\
&+\alpha\Bigl|\nabla_{e_1}(t)u
 - \frac{1}{\alpha(m-1)}f\,c(e^1)\sigma u\Bigr|^2\\
&+\underbrace{\Bigl(\frac{1}{m(m-1)}
 - \frac{1}{\alpha(m-1)^2}\Bigr)}_{:=\,\alpha_1} f^2|u|^2.
\end{aligned}
\]
Therefore, we obtain 
\[
|\mathcal{P}(t)u|^2 \geq \frac{1}{4c}|\nabla_{e_1}(t)u|^2 + \alpha_1 f^2|u|^2.
\]

By Kato's inequality, \(|\nabla(t)u| \geq |\nabla|u||\) almost everywhere. Hence, at any point where \(|u|\) is differentiable, we may choose a local orthonormal frame with $e_1$ given by the unit gradient such that $|\nabla_{e_1} |u||=|\nabla |u| |$ (if \(\nabla|u| \neq 0\), and otherwise the inequality holds trivially). 
Thus, for each $t\in [0,1]$ we have
\begin{equation}\label{eq:Kato_Penrose}
    |\mathcal{P}(t) u|^2 \geq \frac{1}{4c} |d |u||^2 + \alpha_1 f^2 |u|^2
\end{equation}
holds almost everywhere on $M$. 

Green's formula provides the following identity (cf.~\cite{HKKZ24}*{(3.12)})
\begin{equation}\label{eq:identity_for_spectral_estimates}
\int_{\partial M} \langle c(\nu^*) u, f\sigma u\rangle\,dA = -\int_M \bigl(2f^2|u|^2 + \langle u, c(d f)\sigma u\rangle\bigr)\,dV.
\end{equation}

On the boundary, the local condition $s\,c(\nu^*)\sigma u=u$ gives
\[
    \langle c(\nu^*)u,f\sigma u\rangle=-sf|u|^2.
\]
Consequently, \eqref{eq:identity_for_spectral_estimates} is equivalently
\begin{equation}\label{eq:boundary_f_identity}
    \int_{\partial M}sf|u|^2\,dA
    =\int_M\bigl(2f^2|u|^2+\langle u,c(d f)\sigma u\rangle\bigr)\,dV.
\end{equation}
After inserting the Penrose--Kato estimate, adding $(\alpha_2-1)$ times \eqref{eq:boundary_f_identity} produces the coefficients $\alpha_2$ in both the boundary and interior terms below, where $\alpha_2=1-\frac{m\alpha_1}{m-1}$. Explicitly,
\[
1+\frac{m}{m-1}\alpha_1+2(\alpha_2-1)=\alpha_2,
\qquad 1+(\alpha_2-1)=\alpha_2.
\]

Finally, putting \eqref{eq:Penrose_spectral_estimates}, \eqref{eq:Kato_Penrose} and \eqref{eq:identity_for_spectral_estimates} together, for any \(u\in\ker (\mathcal{B}_{f,s}(t))\) $(t\in [0,1])$ we have 
\begin{equation}\label{eq:spectral_flow_connection_spectral_estimate}
\begin{aligned}
\int_{\partial M} \Big( \alpha_2 sf - \frac{m}{2} H \Big) |u|^2 dA
\geq & \frac{m}{m-1} \int_M \Big( \frac{1}{4c} |d |u||^2 + \langle u, \mathscr{R}(t) u \rangle \Big) dV \\
     & + \int_M  \left \langle u, \alpha_2 f^2 u +  \alpha_2 c(d f) \sigma u \right \rangle  dV,
\end{aligned}
\end{equation}
where $\alpha_2=1-\frac{m\alpha_1}{m-1}$. Since $0<\alpha<\frac m{m-1}$, we have
\[
\alpha_1=\frac{1}{m(m-1)}-\frac{1}{\alpha(m-1)^2}<0,
\qquad \alpha_2>1>0
\]
for every $m\geq2$.

If $\dim M$ is even, we also have analogous spectral estimate for a single Callias operator.
For any \(u\in\ker (\mathcal{B}_{f,s})\),
\begin{equation}\label{eq:even_connection_spectral_estimate}
\begin{aligned}
\int_{\partial M}\Bigl(\alpha_2 sf-\frac{m}{2}H\Bigr)|u|^2\,dA
&\ge \frac{m}{m-1}\int_M\!\Bigl(\frac1{4c}|d |u||^2+\langle u,\mathscr{R}u\rangle\Bigr)dV \\
&\quad + \int_M\bigl\langle u,\;\alpha_2 f^2 u+\alpha_2 c(d f)\sigma u\bigr\rangle\,dV,
\end{aligned}
\end{equation}
where $\mathscr{R}$ is defined in \eqref{eq:curvature_endomorphism}.


\section{Proof of the quantitative theorem}\label{sec:formula_in_codimension_zero}

In this section, we prove Theorem~\ref{mthm:quantitative_Llarull_theorem}. 
We start by establishing an estimate that involves an auxiliary parameter $c > \frac{m-1}{4m}$.

\begin{proposition}\label{pro:spectral_quantitative_Llarull_theorem}
    Let $(M,g_M)$ be an $m$-dimensional ($m\geq 2$) complete connected noncompact Riemannian spin manifold. 
    Let $\Phi: M\to \sphm$ be a smooth area-decreasing map which is locally constant at infinity and of nonzero degree. Assume that $\scalM \geq m(m-1)$ on $\supp(d\Phi)$. 
    If $c > \frac{m-1}{4m}$, then
    \[
          \inf_M \scalM < - \frac{1}{c} \lambda_1(M,g_{M}).
    \] 
\end{proposition}

\begin{proof}
 We argue by contradiction. Assume that
\begin{equation}\label{eq:assumption_noncompact_Llarull}                                                  
    \inf_{M} \scalM \geq -\frac{1}{c} \lambda_1(M,g_M),
\end{equation}
where $c>\frac{m-1}{4m}$.
Set $K_0:=\supp(d\Phi)$, $U:= \{ p \in M : d_p \Phi \neq 0 \}$ and $U_\delta := \{ p \in U : \mathbf{a}(p) < \delta \}$
for some constant $0 < \delta < 1$, where $\mathbf{a}(p)$ is the area contraction constant of $\Phi$ at $p$ as in \eqref{defn:contraction_constant}.
Observe that $U_{\delta}\subset U\subset K_0$, and \(K_0\) is contained in the interior of \(M\).
Since $\Phi$ has nonzero degree, the set $U$ cannot be empty. Because $\Phi$ is locally constant at infinity, $U$ is relatively compact. Choose a connected component $U_0$ of $U$. Its boundary in $M$ is nonempty and is contained in the set where $d\Phi=0$. Hence the area contraction constant tends to $0$ along points of $U_0$ approaching $\partial U_0$. Therefore $U_\delta\cap U_0\neq\varnothing$ for every $\delta>0$.

Let \(\mathcal{W}\) be a compact submanifold of \(M\) with smooth boundary such that its interior contains $K_0$.
Since $K_0$ is compact and $\scalM\geq m(m-1)$ on $K_0$, continuity provides a neighborhood $\mathcal{O}$ of $K_0$ with compact closure in $\mathcal{W}^{\circ}$ and $\scalM\geq r_0>0$ on $\mathcal{O}$. Choose a compact smooth domain $K$ with $K_0\subset\operatorname{int}(K)\subset K\subset\mathcal{O}$.
Let $\mathcal{U}$ be an open neighborhood of $K$ in $\mathcal{W}^{\circ}$ with $\overline{\mathcal{U}}\subset\mathcal{O}$. Choose $0<\scal_0<\min\{r_0,m(m-1)\}$, so that $\scalM\geq\scal_0$ on $\mathcal{U}$.
Since $\Phi$ is locally constant outside $K$ and has nonzero degree, its restriction \(\Phi|_{\mathcal{W}} : \mathcal{W} \to \sphm\) (still denoted by \(\Phi\)) is locally constant near \(\partial\mathcal{W}\) and also of nonzero degree.

Following the argument in the proof of~\cite{CZ24}*{Lemma~5.1}, fix a base point \(* \in \sphm\). 
The complement $\mathcal{W}\setminus K$ has finitely many connected components because $K$ is a compact smooth domain. Since $\Phi$ is locally constant there, there are finitely many distinct points $q_1,\cdots,q_k\in\sphm$ such that $\Phi(\mathcal{W}\setminus K)=\{q_1,\cdots,q_k\}$.

Write $\Omega_i=\Phi^{-1}(q_i)\cap(\mathcal{W}\setminus K)$. Then the $\Omega_i$ are open and $\mathcal{W}\setminus K=\bigsqcup_{i=1}^k\Omega_i$.
By continuity, \(\Phi(\overline{\Omega}_i) = \{q_i\}\), so the closures \(\overline{\Omega}_i\) form a family of pairwise disjoint closed subsets of \(\mathcal{W}\). 

Consequently, we can choose pairwise disjoint open neighborhoods \(V_i\) with \(\overline{\Omega}_i\subset V_i\) and smooth cut-off functions \(\nu_i: \mathcal{W} \to [0,1]\) with $\nu_i=1$ on a neighborhood of \(\overline{\Omega}_i\) and $\nu_i=0$ on \(\mathcal{W}\setminus V_i\). 
Let \(\gamma_i: [0,1] \to \sphm\) be geodesics joining \(\gamma_i(0) = *\) to \(\gamma_i(1) = q_i\). 
We now construct a smooth map \(\Psi: \mathcal{W} \to \sphm\) by (using a trick similar to that in~\cite{LSWZ24+}):
\[
\Psi(p) = 
\begin{cases}
\gamma_i(\nu_i(p)), & \text{if } p \in V_i, \\
*, & \text{if } p \in \mathcal{W} \setminus \bigsqcup_{i=1}^k V_i.
\end{cases}
\]
By construction, $\Psi=\Phi=q_i$ on each $\overline{\Omega}_i$, and $\Lambda^2d\Psi=0$ everywhere. Since the image of $\Psi$ lies in a finite union of geodesic arcs, $\deg(\Psi)=0$ for $m\geq2$.
Because $\partial K$ is disjoint from $K_0=\supp(d\Phi)$, the map $\Phi$ is locally constant on a collar of $\partial K$. With the cutoffs chosen above, $\Psi=\Phi$ on such a collar. Thus the bundle identification extends across the boundary of the GL-pair support.

By construction, $\Psi=\Phi$ on a collar of $\partial\mathcal{W}$, so the doubled map $\Theta$ defined below is smooth.
Let $\Theta:\dbl\mathcal{W}:=\mathcal{W}\cup_{\partial \mathcal{W}} \mathcal{W}^- \to \sphm$ be the smooth map defined by $\Theta|_{\mathcal{W}}=\Phi$ and $\Theta|_{\mathcal{W}^-}=\Psi$, where $\mathcal{W}^-$ is a copy of $\mathcal{W}$ with the opposite orientation.
Since \(\Psi\) has degree zero,
\begin{equation}\label{eq:nonzero_degree}
   \deg(\Theta) = \deg(\Phi) \neq 0. 
\end{equation}

We will separate the argument into the cases where $m$ is even and $m$ is odd.

\paragraph{Case 1.}
Assume first that \(m\) is even. 
Then the bundles $\mathcal{E}:=\Phi^*\mathcal{E}_0$ and $\mathcal{F}:=\Psi^*\mathcal{E}_0$, where $\mathcal{E}_0$ is the positive half spinor bundle over $\sphm$, form a GL-pair over $\mathcal{W}$ with support $K$ (cf.~\cite{CZ24}*{Lemma~5.1}). 
By Llarull's argument~\cite{Ll98}, which goes back to~\cite{GL83}, $\deg(\Theta)\neq 0$ implies
\begin{equation}\label{eq:relative_index_not_vanishing}
\ind\bigl(\slashed{D}_{\dbl\mathcal{W},\; \Theta^*\mathcal{E}_0}\bigr) \neq 0.  
\end{equation}
By construction, \(\Theta^*\mathcal{E}_0 = V(\mathcal{E},\mathcal{F})\). Because the pair \((\mathcal{E},\mathcal{F})\) is trivial outside \(K\), it extends trivially to all of \(M\). Consequently, from \eqref{eq:relative_index} and \eqref{eq:relative_index_not_vanishing} we have
\begin{equation}\label{eq:nonvanishing_relative_index}
    \indrel(M,\mathcal{E},\mathcal{F})
    = \ind\bigl(\slashed{D}_{\dbl\mathcal{W},V(\mathcal{E},\mathcal{F})}\bigr)
    = \ind\bigl(\slashed{D}_{\dbl\mathcal{W},\Theta^*\mathcal{E}_0}\bigr)
    \neq 0.
\end{equation}

Denote by \(S\) the relative Dirac bundle over $M$ associated with \((\mathcal{E},\mathcal{F})\) 
and by $\mathcal{D}$ the corresponding Dirac operator on $S$ (see~\cite{CZ24}*{Example~2.5}). 
The curvature endomorphism in the Bochner--Lichnerowicz--Schr\"{o}dinger--Weitzenb\"ock formula (cf.~\cite{LM89}) is given by
\begin{equation}\label{eq:curvature_endomorphism_even}
    \mathscr{R}= \frac{1}{4} \scalM + \mathscr{R}^{\mathcal{E}\oplus\mathcal{F}},
\end{equation}
where $\mathscr{R}^{\mathcal{E}\oplus\mathcal{F}}$ is an even endomorphism of the bundle $S$ depending linearly on the curvature tensor of the connection $\nabla^{\mathcal{E}\oplus\mathcal{F}}$.
Note also that $\mathcal{F}$ is a flat bundle since $\Psi$ induces the zero map on $2$-vectors. 
Together with~\cite{Ll98}*{(4.6)}, we obtain that for each point $p\in M$,
\begin{equation}\label{eq:curvature_term_estimate_area_decreasing}
   \mathscr{R}_p^{\mathcal{E}\oplus\mathcal{F}}\geq - \mathbf{a}(p) \cdot \frac{m(m-1)}{4},
\end{equation}
where $\mathbf{a}(p)$ is the area contraction constant of $\Phi$ at $p$ defined in \eqref{defn:contraction_constant}.

Let \(\psi\colon M\to[0,1]\) be a smooth cut-off function with \(\psi=0\) on \(K\) and \(\psi=1\) outside \(\overline{\mathcal{U}}\).
Choose $\varepsilon>0$ sufficiently small so that
\[
 \varepsilon\alpha_2\sup_{\overline{\mathcal{U}}\setminus U}|d\psi|
 < \frac{m\,\scal_0}{4(m-1)}.
\]
Set \(f:=\varepsilon\psi\). Then \(f\) is admissible, with \(f|_K=0\) and \(f|_{M\setminus\overline{\mathcal{U}}}=\varepsilon\).

Consider the Callias operator \(\mathcal{B}_f=\mathcal{D}+f\sigma\). By the splitting theorem for the index of Callias operators~\cite{Rad94}*{Proposition~2.3} (also see~\cite{CZ24}*{Theorem~3.6}),
\begin{equation}\label{eq:to_splitting}
    \ind(\mathcal{B}_{f})
    = \ind(\mathcal{B}_{f,-1}^{\mathcal{W}})
    + \ind(\mathcal{B}_{f,-1}^{ \overline{M\setminus \mathcal{W}} }).
\end{equation}
where $\mathcal{B}_{f,-1}^{\mathcal{W}}$ and $\mathcal{B}_{f,-1}^{ \overline{M\setminus \mathcal{W}} }$ denote the corresponding Callias operator on $\mathcal{W}$ and $\overline{M\setminus \mathcal{W}}$, respectively.

Since $\mathcal{B}_{f,-1}^{ \overline{M\setminus \mathcal{W}} }$ is a Callias operator on a relative Dirac bundle with empty support and the sign $s=-1$ on all of $\partial (\overline{M\setminus \mathcal{W}})$ (corresponding to $s=+1$ in the inward-normal convention of~\cite{CZ24}), by~\cite{CZ24}*{Lemma~3.7} we have
\begin{equation}\label{eq:vanishing_to_at_infinity}
            \ind(\mathcal{B}_{f,-1}^{ \overline{M\setminus \mathcal{W}} }) = 0.
\end{equation}
By~\cite{CZ24}*{Corollary~3.9}, we have  
    \begin{equation}\label{eq:correspond}
            \ind(\mathcal{B}_{f,-1}^{\mathcal{W}}) = \indrel(M,\mathcal{E},\mathcal{F}).     
    \end{equation}
    
Combining \eqref{eq:nonvanishing_relative_index}, \eqref{eq:to_splitting}, \eqref{eq:vanishing_to_at_infinity} and \eqref{eq:correspond}, we conclude that \(\ind(\mathcal{B}_f)\neq0\).
Then there exists a nonzero element $u\in \ker(\mathcal{B}_{f})$. By Lemma~\ref{lem:kernel_regularization}, $u\in H^1(M,S)$ and $|u|\in H^1(M)$, and the compactly supported integral estimates extend to this kernel section. Using the spectral estimate \eqref{eq:even_connection_spectral_estimate} and \eqref{eq:curvature_endomorphism_even}, we obtain
\[
\begin{aligned}
0 \geq \; & \frac{m}{m-1} \int_M
   \Bigl(\frac{1}{4c} |d |u||^2 + \frac{1}{4}\,\scalM|u|^2 +\bigl\langle u,\mathscr{R}^{\mathcal{E}\oplus\mathcal{F}} u\bigr\rangle\Bigr)\,dV\\
     & + \int_M \left\langle u,\alpha_2 f^2 u
       +\alpha_2 c(d f) \sigma u\right\rangle\,dV,
\end{aligned}
\]
where \(\alpha_2=\alpha_2(c)>1\) is the constant defined in the spectral estimates above.

Recall that \(f = \varepsilon\psi\) and \(\operatorname{supp}(d\psi) \subset \overline{\mathcal{U}}\setminus\operatorname{int}(K)\).  
Since \(\mathscr{R}^{\mathcal{E}\oplus\mathcal{F}} = 0\) on \(M\setminus U\), we split the integrals over \(U\) and \(M\setminus U\).  
Using \eqref{defn:bottom_spectrum},
\[
\begin{aligned}
0 \geq \; & \frac{m}{4c(m-1)}\,\lambda_1(M,g_M)\,\|u\|_{L^2(M)}^2
\\
          & + \frac{m}{m-1} \int_{U}\!\Bigl(\frac{1}{4}\,\scalM|u|^2
             + \bigl\langle u,\mathscr{R}^{\mathcal{E}\oplus\mathcal{F}} u\bigr\rangle\Bigr)\,dV\\
          & + \Bigl(\frac{m \,\scal_0}{4(m-1)}
             - \varepsilon\alpha_2 \sup_{\overline{\mathcal{U}}\setminus U}|d\psi|\Bigr)
             \,\|u\|_{L^2(\overline{\mathcal{U}}\setminus U)}^2\\
          & + \int_{M\setminus\overline{\mathcal{U}}}\!
             \Bigl(\frac{m \,\scalM}{4(m-1)} + \varepsilon^2 \alpha_2 \Bigr)|u|^2\,dV,
\end{aligned}
\]
where for $A\subset M$, the symbol $\|u\|_{L^2(A)}^2$ denotes $\int_A |u|^2 dV$.

On \(U\) we have \(\scalM \ge m(m-1)\) by hypothesis.  
Because \(\Phi\) is area-decreasing, its area contraction constant satisfies \(\mathbf{a}(p)\le 1\) at every \(p\in M\). Hence
\[
\scalM(p) \ge \mathbf{a}(p)\cdot m(m-1), \quad \text{for all } p\in U.
\]
  Applying the curvature estimate \eqref{eq:curvature_term_estimate_area_decreasing} yields
\begin{equation}\label{eq:spectral_Llarull_theorem_even_case}
\begin{aligned}
0 \geq \; & \frac{m}{4c(m-1)}\, \lambda_1(M,g_M)\,\|u\|_{L^2(M)}^2
\\
          & + \frac{m}{4(m-1)}\int_{U}
             \underbrace{\Bigl(\scalM - \mathbf{a}(p)\,m(m-1)\Bigr)}_{\ge 0}\,|u|^2\,dV\\
          & + \Bigl(\frac{m \,\scal_0}{4(m-1)}
             - \varepsilon\alpha_2 \sup_{\overline{\mathcal{U}}\setminus U}|d\psi|\Bigr)
             \,\|u\|_{L^2(\overline{\mathcal{U}}\setminus U)}^2\\
          & + \Bigl(\frac{m}{4(m-1)}\inf_{M\setminus\overline{\mathcal{U}}}\scalM
             + \varepsilon^2 \alpha_2 \Bigr)\,
             \|u\|_{L^2(M\setminus\overline{\mathcal{U}})}^2.
\end{aligned}
\end{equation}

  By the preceding choice of $\varepsilon$, the coefficient of $\|u\|_{L^2(\overline{\mathcal U}\setminus U)}^2$ is strictly positive.
  Since \(M\) is noncompact whereas \(\overline{\mathcal U}\) is compact, the set
  \(M\setminus\overline{\mathcal U}\) is nonempty. Moreover, \(\lambda_1(M,g_M)\geq0\) by its Rayleigh-quotient definition.

Splitting the first term according to the decomposition \(M = \mathcal{U} \cup (M\setminus\overline{\mathcal{U}})\) and using the lower bound \(\scalM - \mathbf{a}(p) \cdot m(m-1) \ge (1-\delta)m(m-1) > 0\) on \(U_\delta \subset U\), we obtain from \eqref{eq:spectral_Llarull_theorem_even_case}
\[
\begin{aligned}
0 \geq \; & \frac{m}{4c(m-1)}\,\lambda_1(M,g_M)\,\|u\|_{L^2(\mathcal{U})}^2
\\
          & + \underbrace{\frac{1}{4}(1-\delta)m^2}_{>0}\,
             \|u\|_{L^2(U_\delta)}^2\\
          & + \underbrace{\Bigl(\frac{m\, \scal_0}{4(m-1)}
             - \varepsilon \alpha_2 \sup_{\overline{\mathcal{U}}\setminus U}|d\psi|\Bigr)}_{>0}\,
             \|u\|_{L^2(\overline{\mathcal{U}}\setminus U)}^2\\
          & + \underbrace{\Bigl(\frac{m}{4(m-1)}
             \inf_{M\setminus\overline{\mathcal{U}}}\scalM
             + \frac{m}{4c(m-1)}\lambda_1(M,g_M)
             + \varepsilon^2 \alpha_2 \Bigr)}_{>0\ \text{by } \eqref{eq:assumption_noncompact_Llarull}}\,
             \|u\|_{L^2(M\setminus\overline{\mathcal{U}})}^2.
\end{aligned}
\]
All terms on the right-hand side are nonnegative with some strictly positive coefficients. Consequently, \(u\) must vanish on \(U_\delta \cup (\overline{\mathcal{U}}\setminus U) \cup (M\setminus\overline{\mathcal{U}}) = U_\delta \cup (M\setminus U)\). In particular, \(u = 0\) on the open set \(U_\delta\).
  Elliptic regularity for the Callias operator gives the regularity needed for unique continuation. Squaring the Callias equation gives a second-order elliptic equation with bounded lower-order coefficients, so Aronszajn's unique continuation theorem~\cite{Aron57} applies. Since $M$ is connected, it forces $u\equiv 0$ on $M$ -- a contradiction.

\medskip

\paragraph{Case 2.}
Now assume that $m$ is odd.
Following~\cite{Ge93}*{page~493} (also see~\cite{LSW24}*{Section~2.3}), one can construct a trivial bundle $\mathcal{E}_0$ over $\sphm$ equipped with the trivial connection $\nabla^0$, together with a smooth map $\rho_0 \in C^{\infty}(\sphm, U(l))$, such that
\[
\nabla^{\mathcal{E}_0}(t) := \nabla^0 + t\, \rho_0^{-1}[\nabla^0, \rho_0], \quad t \in [0,1],
\]
defines a smooth family of Hermitian connections on $\mathcal{E}_0$.
In this setting, the odd Chern character form associated to $\rho_0$ and $\nabla^0$, defined by (\cite{Ge93}, also see~\cite{Zh01}*{(1.50)})
\[
\begin{aligned}
\ch(\rho_0,\nabla^0)
 := \sum_{j=0}^{\infty}
 \left(\frac{1}{2\pi\sqrt{-1}} \right)^{j+1}
 \frac{j!}{(2j+1)!} \cdot\operatorname{tr}
 \big[(\rho_0^{-1}[\nabla^0,\rho_0])^{2j+1}\big],
\end{aligned}
\]
satisfies (cf.~\cite{Ge93}*{page~493}, also see~\cite{LSW24}*{page~1108})
\begin{equation}\label{eq:odd_Chern_character_on_sphere}
    \int_{\sphm} \ch(\rho_0,\nabla^0) =-1.
\end{equation}

Since the aforementioned maps $\Phi=\Psi$ on $\overline{\mathcal{W}\setminus K}$, the bundles $\mathcal{E}:=\Phi^*\mathcal{E}_0$, $\mathcal{F}:=\Psi^*\mathcal{E}_0$ form a GL-pair over $\mathcal{W}$ with support $K$.
As in~\cite{Shi24+}, set $\rho^+=\Phi^*\rho_0$ and $\rho^-=\Psi^*\rho_0$.
Since $\Phi=\Psi$ on a collar of $\partial \mathcal{W}$ and is locally constant there, $\rho^+=\rho^-$ is locally constant on the same collar. Then $\rho^+$ and $\rho^-$ can be glued smoothly to yield $\widetilde{\rho}\in C^{\infty}(\dbl\mathcal{W}, U(l))$.

By definition, we have $\Theta^*\rho_0=\widetilde{\rho}$.
The pull-back bundle $\Theta^*\mathcal{E}_0=V(\mathcal{E},\mathcal{F})$ is a Hermitian vector bundle over $\dbl\mathcal{W}$ with a family of Hermitian connections induced by $\Theta^*\rho_0$. 
Let $\slashed{S}_{\dbl\mathcal{W}}$ be the complex spinor bundle of $\dbl\mathcal{W}$ and let $\slashed{D}_{\Theta^*\mathcal{E}_0}$ be the spin Dirac operator on $\dbl\mathcal{W}$ twisted by $\Theta^*\mathcal{E}_0$ with connection
\[
\nabla^{\slashed{S}_{\dbl\mathcal{W}}\otimes \Theta^*\mathcal{E}_0}(t)
=\nabla^{\slashed{S}_{\dbl\mathcal{W}}}\otimes\operatorname{id} +\operatorname{id}\otimes\Theta^*\nabla^{\mathcal{E}_0}(t),
\qquad t\in[0,1].
\]
For the operators on $M$ used below, write the curvature endomorphism in \eqref{eq:BLW} as
\begin{equation}\label{eq:odd_curvature_endomorphism}
    \mathscr{R}_M(t)=\frac{1}{4} \scalM + \mathscr{R}^{\mathcal{E}\oplus\mathcal{F}}(t).
\end{equation}
Here $\mathscr{R}^{\mathcal{E}\oplus\mathcal{F}}(t)$ is obtained from the two pull-back connections $\Phi^*\nabla^{\mathcal E_0}(t)$ and $\Psi^*\nabla^{\mathcal E_0}(t)$. On the double, the corresponding endomorphism is
\[
    \mathscr{R}_{\dbl\mathcal{W}}(t)=\frac{1}{4}\scal_{g_{\dbl\mathcal{W}}}+\mathscr{R}^{\Theta^*\mathcal{E}_0}(t),
\]
where $g_{\dbl\mathcal{W}}$ is the chosen metric on the double and the twisting curvature is computed from the pull-back connection $\Theta^*\nabla^{\mathcal{E}_0}(t)$.

It is proved in~\cite{LSW24}*{page~1108} that
for each $t\in [0,1]$ and each $p\in M$,
\begin{equation}\label{eq:curvature_term_area_decreasing}
   \mathscr{R}_p^{\mathcal{E}\oplus\mathcal{F}}(t) \geq -\mathbf{a}(p) \cdot \frac{m(m-1)}{4},
\end{equation}
where $\mathbf{a}(p)$ is the area contraction constant of $\Phi$ at $p$ defined in \eqref{defn:contraction_constant}.

In this construction $\mathcal{E}$ and $\mathcal{F}$ are trivial bundles. Extend $\Phi$ and $\Psi$ over $M\setminus\mathcal W$ by their common locally constant values. Define $\rho_M^+=\Phi^*\rho_0$, $\rho_M^-=\Psi^*\rho_0$ and $\rho_M=\rho_M^+\oplus\rho_M^-$. The two blocks agree outside $K$ and are locally constant there, so they commute with $\sigma$ wherever $f$ may be nonzero.
Thus, combining \eqref{eq:relative_spectral_flow} with~\cite{Ge93}*{Theorem 2.8}, we obtain
\begin{equation}\label{eq:non-vanishing_relative_spectral_flow}
\begin{aligned}
        \sfrel(M,\mathcal{E},\mathcal{F},\rho_M)
        =& \sflow (\slashed{D}_{\dbl\mathcal{W},V(\mathcal{E},\mathcal{F})},\widetilde{\rho}) \\
        =& \sflow (\slashed{D}_{\dbl\mathcal{W},\Theta^*\mathcal{E}_0},\Theta^*\rho_0) \\
        =& -\int_{\dbl\mathcal{W}} \Af(\dbl\mathcal{W})\wedge \Theta^*\ch(\rho_0,\nabla^0) \\
        =& -\int_{\dbl\mathcal{W}}  \Theta^*\ch(\rho_0,\nabla^0)_{[m]} \\
        =& -\deg(\Theta) \int_{\sphm} \ch(\rho_0,\nabla^0) \\
        =& \deg(\Theta) \neq 0,
\end{aligned}
\end{equation}
where $\Af(\dbl\mathcal{W})$ is the $\widehat{A}$-form of $\dbl\mathcal{W}$. The passage to the degree-$m$ component uses that the cohomology of \(\sphm\) is trivial in all degrees except \(0\) and \(m\). The next equality is the degree formula for pullbacks, and the last equality follows from \eqref{eq:odd_Chern_character_on_sphere}. The final nonvanishing uses \eqref{eq:nonzero_degree}.

Denote by \(S\) the relative Dirac bundle over $M$ associated with \((\mathcal{E},\mathcal{F})\) 
and by $\mathcal{D}$ the corresponding Dirac operator on $S$ (see~\cite{CZ24}*{Example~2.5}). 
Let \(\psi\colon M\to[0,1]\) be a smooth cut-off function with \(\psi=0\) on \(K\) and \(\psi=1\) outside \(\overline{\mathcal{U}}\).
Choose $\varepsilon>0$ sufficiently small so that
\[
 \varepsilon\alpha_2\sup_{\overline{\mathcal{U}}\setminus U}|d\psi|
 < \frac{m\,\scal_0}{4(m-1)}.
\]
Set \(f:=\varepsilon\psi\). Then \(f\) is admissible, with \(f|_K=0\) and \(f|_{M\setminus\overline{\mathcal{U}}}=\varepsilon\).

We consider a family of Callias operators $\mathcal{B}_{f}(t)=(1-t)\mathcal{B}_{f} + t \rho_M^{-1} \mathcal{B}_{f} \rho_M$ for $t\in [0,1]$, where \(\mathcal{B}_f=\mathcal{D}+f\sigma\). 
By the splitting theorem for the corresponding spectral flow~\cite{Shi24+}*{Theorem~2.10}, we have 
\begin{equation}\label{eq:to_splitting_odd}
    \sflow(\mathcal{B}_{f},\rho_M)
    = \sflow(\mathcal{B}_{f,-1}^{\mathcal{W}},\rho)
    + \sflow\bigl(\mathcal{B}_{f,-1}^{ \overline{M\setminus \mathcal{W}} },
       \rho_{ \overline{M\setminus \mathcal{W}} }\bigr).
\end{equation}
where $\rho:=\rho_M|_{\mathcal W}$ and $\rho_{ \overline{M\setminus \mathcal{W}} }:=\rho_M|_{\overline{M\setminus\mathcal W}}$ are restrictions of the same block automorphism.
Since $\mathcal{B}_{f,-1}^{ \overline{M\setminus \mathcal{W}} }$ is a Callias operator on a relative Dirac bundle with empty support and the sign $s=-1$ on all of $\partial(\overline{M\setminus \mathcal{W}})$ (corresponding to $s=+1$ in the inward-normal convention of~\cite{Shi24+}),~\cite{Shi24+}*{Lemma~3.4} then implies
\begin{equation}\label{eq:vanishing_to_at_infinity_odd}
   \sflow(\mathcal{B}_{f,-1}^{ \overline{M\setminus \mathcal{W}} },\rho_{ \overline{M\setminus \mathcal{W}} } )=0.
\end{equation}

By~\cite{Shi24+}*{Theorem~3.10}, we have  
    \begin{equation}\label{eq:correspond_odd}
            \sflow(\mathcal{B}_{f,-1}^{\mathcal{W}},\rho) = \sfrel(M,\mathcal{E},\mathcal{F},\rho_M).
    \end{equation}
    
Using \eqref{eq:non-vanishing_relative_spectral_flow}, \eqref{eq:to_splitting_odd}, \eqref{eq:vanishing_to_at_infinity_odd} and \eqref{eq:correspond_odd}, we obtain $ \sflow(\mathcal{B}_{f},\rho_M)  \neq 0$.
Thus, there exist some $t\in [0,1]$ and a nonzero $u\in \ker(\mathcal{B}_{f}(t))$. By Lemma~\ref{lem:kernel_regularization}, applied uniformly to the odd-dimensional family, $u\in H^1(M,S)$ and $|u|\in H^1(M)$, so the compactly supported integral estimates extend to $u$. By \eqref{eq:spectral_flow_connection_spectral_estimate} and \eqref{eq:odd_curvature_endomorphism}, 
\[
\begin{aligned}
0 \geq \; & \frac{m}{m-1} \int_M
   \Bigl(\frac{1}{4c} |d |u||^2 + \frac{1}{4}\,\scalM|u|^2 +\bigl\langle u,\mathscr{R}^{\mathcal{E}\oplus\mathcal{F}}(t) u\bigr\rangle\Bigr)\,dV\\
     & + \int_M \left\langle u,\alpha_2 f^2 u
       +\alpha_2 c(d f) \sigma u\right\rangle\,dV.
\end{aligned}
\]

Recall that $f=\varepsilon\psi$ and $\supp(d\psi)\subset \overline{\mathcal{U}}\setminus\operatorname{int}(K)$. Note that $\mathscr{R}^{\mathcal{E}\oplus\mathcal{F}}(t)=0$ on $M\setminus U$.
Splitting the integrals over \(U\) and \(M\setminus U\) and using \eqref{defn:bottom_spectrum}, we obtain
\[
\begin{aligned}
0 \geq \; & \frac{m}{4c(m-1)}\,\lambda_1(M,g_M)\,\|u\|_{L^2(M)}^2
\\
          & + \frac{m}{m-1} \int_{U}\!\Bigl(\frac{1}{4}\,\scalM|u|^2
             + \bigl\langle u,\mathscr{R}^{\mathcal{E}\oplus\mathcal{F}}(t) u\bigr\rangle\Bigr)\,dV\\
          & + \Bigl(\frac{m \,\scal_0}{4(m-1)}
             - \varepsilon\alpha_2 \sup_{\overline{\mathcal{U}}\setminus U}|d\psi|\Bigr)
             \,\|u\|_{L^2(\overline{\mathcal{U}}\setminus U)}^2\\
          & + \int_{M\setminus\overline{\mathcal{U}}}\!
             \Bigl(\frac{m \,\scalM}{4(m-1)} + \varepsilon^2 \alpha_2 \Bigr)|u|^2\,dV.
\end{aligned}
\]

On \(U\) we have \(\scalM \ge m(m-1)\) by hypothesis.  
Since \(\Phi\) is area-decreasing, its area contraction constant satisfies \(\mathbf{a}(p)\le 1\) everywhere. Consequently
\[
\scalM(p) \ge \mathbf{a}(p)\cdot m(m-1) \quad \text{for all } p\in U.
\]
Inserting the curvature estimate \eqref{eq:curvature_term_area_decreasing} yields
\[
\begin{aligned}
0 \geq \; & \frac{m}{4c(m-1)}\, \lambda_1(M,g_M)\,\|u\|_{L^2(M)}^2
\\
          & + \frac{m}{4(m-1)}\int_{U}
             \underbrace{\Bigl(\scalM - \mathbf{a}(p)\,m(m-1)\Bigr)}_{\ge 0}\,|u|^2\,dV\\
          & + \Bigl(\frac{m \,\scal_0}{4(m-1)}
             - \varepsilon\alpha_2 \sup_{\overline{\mathcal{U}}\setminus U}|d\psi|\Bigr)
             \,\|u\|_{L^2(\overline{\mathcal{U}}\setminus U)}^2\\
          & + \Bigl(\frac{m}{4(m-1)}\inf_{M\setminus\overline{\mathcal{U}}}\scalM
             + \varepsilon^2 \alpha_2 \Bigr)\,
             \|u\|_{L^2(M\setminus\overline{\mathcal{U}})}^2.
\end{aligned}
\]

By the preceding choice of $\varepsilon$, the coefficient of $\|u\|_{L^2(\overline{\mathcal U}\setminus U)}^2$ is strictly positive.
The remaining argument now proceeds exactly as in the even-dimensional case (see the discussion following \eqref{eq:spectral_Llarull_theorem_even_case}), leading to a contradiction.

This completes the proof of the proposition.
\end{proof}

We now apply Proposition~\ref{pro:spectral_quantitative_Llarull_theorem} to prove Theorem~\ref{mthm:quantitative_Llarull_theorem}.

\begin{proof}[Proof of Theorem~\ref{mthm:quantitative_Llarull_theorem}]
By Proposition~\ref{pro:spectral_quantitative_Llarull_theorem}, the inequality
\begin{equation}\label{eq:temp}
\inf_M \scalM < -\frac{1}{c} \lambda_1(M,g_M)
\end{equation}
holds for every $c > \frac{m-1}{4m}$. Observe that $\lambda_1(M,g_M) \ge 0$. We split the proof into two cases.

\paragraph{Case 1.} Assume that $\lambda_1(M,g_M)=0$.
Choosing any $c_0 > \frac{m-1}{4m}$ (for instance, $c_0 = 1$) in \eqref{eq:temp} gives
\[
\inf_M \scalM < 0 .
\]

\paragraph{Case 2.} Assume that $\lambda_1(M,g_M)>0$.
Because \eqref{eq:temp} is valid for all $c > \frac{m-1}{4m}$,
the quantity $\inf_M \scalM$ is a lower bound for the set
$\bigl\{-\frac{1}{c}\lambda_1(M,g_M) \colon c > \frac{m-1}{4m}\bigr\}$.
The map $c \mapsto -\frac{1}{c}\lambda_1(M,g_M)$ is continuous and strictly increasing on this interval,
hence its infimum equals the limit as $c \to (\frac{m-1}{4m})^+$, that is,
\[
\inf_{c > \frac{m-1}{4m}} \Bigl\{ - \frac{1}{c}\lambda_1(M,g_M) \Bigr\}
= - \frac{4m}{m-1}\lambda_1(M,g_M).
\]
Since $\inf_M \scalM$ is a lower bound for this set, the definition of the infimum yields
\[
\inf_M \scalM \le - \frac{4m}{m-1} \lambda_1(M,g_M),
\]
which is precisely the desired inequality.
\end{proof}


\section{An illustrative example and sharpness}\label{sec:sharpness_quantitative_Llarull}

We begin with a three-dimensional model that illustrates the nontriviality of Theorem~\ref{mthm:quantitative_Llarull_theorem}. The purpose of this example is to exhibit a concrete map and metric satisfying the hypotheses. The optimality of the coefficient is established by the all-dimensional construction that follows.

\begin{example}\label{example:Llarull_theorem}
Let $N=\mathbf{T}^2\times\mathbf R$ carry the warped product metric
$g_N=\cosh^2(t)g_{\mathbf{T}^2}+dt^2$, where $g_{\mathbf{T}^2}$ is flat
(cf.~\cite{MWang24}*{Example~1.5}). The warped product formula gives
\[
 \scal_{g_N}=-4\frac{(\cosh t)''}{\cosh t}
 -2\left(\frac{(\cosh t)'}{\cosh t}\right)^2
 =-6+2\operatorname{sech}^2t.
\]
To compute the bottom spectrum, put $v=\cosh(t)u$. For
$u\in C_c^\infty(N)$, integration by parts yields
\[
\begin{aligned}
 \int_N|du|_{g_N}^2\,dV_{g_N}
 &=\int_{\mathbf{T}^2\times\mathbf R}
  \left(|\partial_tv|^2+|v|^2+\cosh^{-2}t\,|d_{\mathbf{T}^2}v|^2\right)
  \,dV_{g_{\mathbf{T}^2}}\,dt,\\
 \int_Nu^2\,dV_{g_N}
 &=\int_{\mathbf{T}^2\times\mathbf R}|v|^2\,dV_{g_{\mathbf{T}^2}}\,dt.
\end{aligned}
\]
Thus $\lambda_1(N,g_N)\ge1$. Taking $v$ constant on $\mathbf{T}^2$ and supported
on intervals whose lengths tend to infinity gives the reverse inequality,
so $\lambda_1(N,g_N)=1$.
Let $\sph^1$ be a standard circle with base point $p$. Choose a smooth
$\eta:\mathbf R\to[0,1]$ with $\eta=0$ on $(-\infty,-1]$ and $\eta=1$ on
$[1,\infty)$, and set $\varphi(t)=\exp(2\pi i\eta(t))$. It is constant
outside $(-1,1)$ and has degree one after the two ends of $\mathbf R$ are
identified.
Fix a smooth degree-one map $H:\mathbf{T}^2\to\sph^2$. If
$\mathbf{T}^2=\mathbf R^2/\Lambda$, let $\widehat{\mathbf{T}}^2=\mathbf R^2/(n\Lambda)$ with
its natural $n^2$-sheeted covering $p_n([x])=[x]\in\mathbf{T}^2$, and give
$\widehat{\mathbf{T}}^2$ the lifted flat metric. Define the degree-one map
$q_n:\widehat{\mathbf{T}}^2\to\mathbf{T}^2$ by $q_n([x])=[x/n]$. Then
$\operatorname{Lip}(q_n)=n^{-1}$, so $h=H\circ q_n$ has nonzero degree and
arbitrarily small Lipschitz constant for large $n$. Given
$0<\delta\le1$, choose $n$ so that $\operatorname{Lip}(h)\le\delta$.
Let $M= \widehat{\mathbf{T}}^2 \times \mathbf{R}$ with the lifted metric $g_0:= \cosh^2(t) g_{\widehat{\mathbf{T}}^2} + dt^2$. 
Again $\scal_{g_0}=-6+2\cosh^{-2}(t)$ and $\lambda_1(M,g_0)=1$. The latter
follows from the same ground-state transform as above, applied to the
finite flat torus $\widehat{\mathbf{T}}^2$.
We consider the chain of maps
\[
    M= \widehat{\mathbf{T}}^2 \times \mathbf{R} \overset{h\times \varphi}{\to} \sph^2 \times \sph^1 \overset{\phi}{\to} \sph^3,
\]
where $\phi:\sph^2\times\sph^1\to\sph^3$ is a smooth degree-one map
constant on a neighborhood of
$(\sph^2\times\{p\})\cup(\{q\}\times\sph^1)$ for some $q\in\sph^2$.
To construct it, choose an embedded closed $3$-ball in the complement of
this wedge, map the ball to $\sph^3$ with degree one while making the map
constant near its boundary, and extend that constant map over the
complement. The resulting map factors through the smash product.
Let $\Phi=\phi \circ (h\times \varphi)$. Since $\varphi$ and $\phi$ are
constant near the relevant ends, $\supp(d\Phi)\subset
K:=\widehat{\mathbf{T}}^2\times[-1,1]$ and $\deg(\Phi)=\deg(h)\ne0$.
The horizontal and vertical singular values of $h\times\varphi$ are
bounded by $\delta/\cosh t$ and $B_0:=\sup|\varphi'|$. Hence
\[
  \|\Lambda^2d\Phi\|_{g_0}\le C\delta
\]
where $C=\|\Lambda^2d\phi\|_\infty\max\{1,B_0\}$ is independent of
$\delta$.
Choose $\chi\in C_c^\infty((-2,2))$ with $0\leq\chi\leq1$ and $\chi=1$ on a neighborhood of $[-1,1]$, and put
\[
 f(x,t)=-\frac54t^2\chi(t).
\]
Then $-5\leq f\leq0$ and $f=0$ for $|t|\geq2$.
Consider the rescaled metric $g_M=e^{2f}g_0$. 
Here $\Delta_{g_0}=\operatorname{div}_{g_0}\nabla$ is the geometric Laplacian (nonpositive on $L^2$), while $\lambda_1$ is the nonnegative Rayleigh bottom defined in \eqref{defn:bottom_spectrum}.
On $K$, $\|\Lambda^2d\Phi\|_{g_M}=e^{-2f}\|\Lambda^2d\Phi\|_{g_0}
\le e^{5/2}C\delta$. Thus $\delta\le e^{-5/2}C^{-1}$ makes
$\Phi:(M,g_M)\to\sph^3$ area-decreasing.
Moreover, 
\[
  \begin{aligned}
      \scal_{g_M}  =& e^{-2f} (\scal_{g_0} - 4 \Delta_{g_0} f - 2|\nabla f|_{g_0}^2) \\
               =& e^{\frac{5}{2} t^2} \Big( 4+2\cosh^{-2}(t)+20t \tanh(t)-\frac{25}{2}t^2 \Big) \\
            \geq & 6 \quad \text{on} \ \ K,
  \end{aligned}
\]
For $s=|t|\le1$, $\tanh s\ge s\tanh1>\frac34s$ and
$\tanh^2s\le s^2$, so the expression in parentheses is at least
$6+\frac12s^2\ge6$. Since $f=0$ for $|t|\ge2$, $\scal_{g_M}=\scal_{g_0}$ there.
Moreover $e^{-10}g_0\le g_M\le g_0$, so $g_M$ is complete.
\[
\begin{aligned}
    \lambda_1(M,g_M)
    &=\inf_{u\in C_c^\infty(M)\setminus\{0\}}
     \frac{\int_M e^f|du|_{g_0}^2\,dV_{g_0}}
          {\int_M e^{3f}u^2\,dV_{g_0}}
    \\
    &\ge e^{-5}\lambda_1(M,g_0)=e^{-5}>0.
\end{aligned}
\]
Thus this is an explicit three-dimensional example satisfying the
hypotheses. The all-dimensional sharpness statement is proved below.
\end{example}

The preceding example gives a concrete noncompact metric with positive bottom spectrum, but it does not determine the optimal coefficient in Theorem~\ref{mthm:quantitative_Llarull_theorem}. We now prove the sharpness proposition stated in the Introduction.

\begin{proof}[Proof of Proposition~\ref{prop:sharpness_quantitative_Llarull}]
Put $d=m-1$ and $\alpha=d/2$. Choose a constant
\[
    K>\frac{d(2d+1)}{4}.
\]
We first choose an even, nonnegative function $\psi\in C_c^\infty(\mathbf R)$ satisfying
\[
    \psi''=-K\quad\text{on }[-1,1],
    \qquad
    \psi''\leq \frac d8\quad\text{on }\mathbf R.
\]
For completeness, put $m_0=d/8$, choose smooth functions
$a:[0,\infty)\to[0,1]$ and $b\in C_c^\infty((0,1))$ with $a=1$ on
$[0,1]$, $a=0$ on $[2,\infty)$, and $0\leq b\leq1$, $b\not\equiv0$.
If $I=\int_0^\infty a$ and $J=\int_0^1 b$, set $\ell=KI/(m_0J)$ and
\[
    q(s)=-Ka(s)+m_0b((s-3)/\ell),
    \qquad
    Q(s)=\int_0^s q(r)\,dr.
\]
The negative part of $q$ precedes its positive part, and
$\int_0^\infty q=0$. Hence $Q\leq0$ and $Q$ vanishes for large $s$.
Moreover $q\leq m_0=d/8$. Choosing $S$ beyond the support of $q$ and defining
\[
    \psi(s)=-\int_s^S Q(r)\,dr\quad (s\geq0),
    \qquad
    \psi(-s)=\psi(s),
\]
we have $\psi'=Q$, $\psi''=q$ on $(0,S)$, while $\psi=0$ near and
beyond $S$. Since $Q\leq0$, $\psi\geq0$. Near $s=0$,
$\psi(s)=\psi(0)-Ks^2/2$, so the even extension is smooth. Thus $\psi$
is even, nonnegative and compactly supported, with the required bounds.
We fix this $\psi$.

For $\varepsilon>0$, set
\[
    w(t)=1+\varepsilon^2\psi(t/\varepsilon),
    \qquad
    V(t)=\alpha^2+\frac{w''(t)}{w(t)}.
\]
Let $F$ solve
\[
    F''=VF,
    \qquad F(0)=1,
    \qquad F'(0)=0.
\]
Writing $S>0$ for a number such that $\supp(\psi)\subset[-S,S]$ and
$T_\varepsilon:=\varepsilon S$, we have $|V|\leq C_0$ with
$C_0=\alpha^2+\|\psi''\|_{L^\infty}$. For $0\leq t\leq T_\varepsilon$,
\[
 F(t)=1+\int_0^t(t-s)V(s)F(s)\,ds.
\]
If $M_\varepsilon=\sup_{[0,T_\varepsilon]}|F|$ and $C_0T_\varepsilon^2\leq1$, then
$M_\varepsilon\leq1+\frac12C_0T_\varepsilon^2M_\varepsilon$, hence $M_\varepsilon\leq2$. Consequently, for
sufficiently small $\varepsilon$,
\[
    \sup_{|t|\leq T_\varepsilon}|F(t)|\leq2,
    \qquad \sup_{|t|\leq T_\varepsilon}|F(t)-1|\leq C_0T_\varepsilon^2,
    \qquad \sup_{|t|\leq T_\varepsilon}|F'(t)|\leq2C_0T_\varepsilon.
\]
Taking $\varepsilon$ small enough that $C_0T_\varepsilon^2<1$ gives $F>0$ on
$[-T_\varepsilon,T_\varepsilon]$. After reducing $\varepsilon$ once more, $F(T_\varepsilon)>0$ and
$|F'(T_\varepsilon)|<\alpha F(T_\varepsilon)$. Since $V=\alpha^2$ outside $[-T_\varepsilon,T_\varepsilon]$, evenness gives
\[
    F(t)=Ae^{\alpha|t|}+Be^{-\alpha|t|}\quad (|t|\geq T_\varepsilon)
\]
where
\[
 A=\frac12e^{-\alpha T_\varepsilon}\left(F(T_\varepsilon)+\frac{F'(T_\varepsilon)}{\alpha}\right)>0,\qquad
 B=\frac12e^{\alpha T_\varepsilon}\left(F(T_\varepsilon)-\frac{F'(T_\varepsilon)}{\alpha}\right)>0.
\]
Hence $F>0$ on $\mathbf R$. In the remainder of this proof, \(f\) denotes the warping factor and is unrelated to the Callias potential used earlier. Define
\[
    f=F^{2/d},
    \qquad
    M=\mathbf{T}^d\times\mathbf R,
    \qquad
    g_L=dt^2+L^2f(t)^2g_{\mathbf{T}^d},
\]
where $(\mathbf{T}^d,g_{\mathbf{T}^d})$ is any flat oriented torus and $L\geq1$ is a constant. The manifold $M$ is parallelizable and therefore spin, and $g_L$ is complete.

Let $h=f'/f$. The scalar curvature of this warped product is
\[
\begin{aligned}
    \scal_{g_L}
    &=-2d\frac{f''}{f}-d(d-1)h^2
      =-4\frac{F''}{F}-dh^2 \\
    &=-d^2-\frac{4\psi''(t/\varepsilon)}{w(t)}-dh^2.
\end{aligned}
\]
This expression is independent of $L$. On $[-T_\varepsilon,T_\varepsilon]$, $h=O(\varepsilon)$ uniformly. After reducing $\varepsilon$ if necessary, $h^2\leq1/2$ there, and the bound on $\psi''$ gives
\[
    \scal_{g_L}\geq-d^2-\frac d2-\frac d2=-d(d+1)
    \quad\text{on }[-T_\varepsilon,T_\varepsilon].
\]
On the ends, $A,B>0$ implies $|F'/F|<\alpha$ and hence $|h|<1$. Therefore
\[
    \scal_{g_L}=-d^2-dh^2>-d(d+1)\quad\text{for }|t|\geq T_\varepsilon,
\]
while $h(t)\to\pm1$ as $t\to\pm\infty$. It follows that
\[
    \inf_M\scal_{g_L}=-d(d+1)=-m(m-1).
\]
On the smaller band $|t|\leq\varepsilon$, we have $\psi''(t/\varepsilon)=-K$. Since $w\to1$ and $h\to0$ uniformly on this band as $\varepsilon\to0$, the choice of $K$ gives, for sufficiently small $\varepsilon$,
\[
    \scal_{g_L}\geq d(d+1)=m(m-1)
    \quad\text{on }\mathbf{T}^d\times[-\varepsilon,\varepsilon].
\]

It remains to construct the map and compute the bottom spectrum. Choose a smoothly embedded ball $B$ with closure contained in $\mathbf{T}^d\times(-\varepsilon,\varepsilon)$. In oriented coordinates on $B$, a radial degree-one map to $\sphm$ can be chosen constant on a collar of $\partial B$ and smooth at the centre. Extending it by that constant over $M\setminus B$ gives a smooth degree-one map $\Phi:M\to\sphm$. Thus
\[
    \supp(d\Phi)\subset \mathbf{T}^d\times(-\varepsilon,\varepsilon).
\]
For $L=1$, let $C=\sup_M\|\Lambda^2d\Phi\|_{g_1}<\infty$. The identity map from $(M,g_L)$ to $(M,g_1)$ has singular values $1$ in the $\mathbf R$ direction and $L^{-1}$ in each torus direction. Its two-dilation is therefore $L^{-1}$ (also when $d=1$). Hence
\[
    \|\Lambda^2d\Phi\|_{g_L}\leq \frac CL.
\]
Choosing $L\geq\max\{1,C\}$ makes $\Phi$ area-decreasing. The scalar
curvature formula above is independent of $L$, and the computation below
shows that the bottom spectrum is independent of $L$ as well.

For the latter, note that $f^d=F^2$. If $u\in C_c^\infty(M)$ and $v=Fu$, integration by parts gives
\[
\begin{aligned}
\int_M|d u|_{g_L}^2\,dV_{g_L}
 =L^d\int_{\mathbf{T}^d\times\mathbf R}
 \left(|\partial_t v|^2+\frac{F''}{F}v^2
 +L^{-2}f^{-2}|d_{\mathbf{T}^d}v|_{g_{\mathbf{T}^d}}^2\right)dV_{g_{\mathbf{T}^d}}\,dt,\\
\int_M u^2\,dV_{g_L}
 =L^d\int_{\mathbf{T}^d\times\mathbf R}v^2\,dV_{g_{\mathbf{T}^d}}\,dt.
\end{aligned}
\]
Since $F''/F=\alpha^2+w''/w$, applying the ground-state identity in the
$t$ variable for each torus point gives
\[
    \int_{\mathbf R}\left(|v'|^2+\frac{w''}{w}v^2\right)dt
    =\int_{\mathbf R}w^2\left|(v/w)'\right|^2dt\geq0
\]
and the nonnegative torus term show that every Rayleigh quotient is at least $\alpha^2$. Conversely, fix
$0\ne\eta\in C_c^\infty((0,1))$, let $a_R=T_\varepsilon+R$, and set
$v_R(x,t)=\eta((t-a_R)/R)$, which is constant on the torus and supported
in the positive end where $w=1$ and $F''/F=\alpha^2$. Then
$u_R=v_R/F$ is compactly supported and
\[
 \frac{\int_M|du_R|_{g_L}^2\,dV_{g_L}}
      {\int_Mu_R^2\,dV_{g_L}}
 =\alpha^2+\frac1{R^2}
 \frac{\int_0^1|\eta'|^2\,ds}{\int_0^1|\eta|^2\,ds}
 \longrightarrow\alpha^2.
\]
Thus
\[
    \lambda_1(M,g_L)=\alpha^2=\frac{(m-1)^2}{4}.
\]
Combining this identity with the scalar-curvature computation and taking $g_M:=g_L$ proves the proposition.
\end{proof}

\begin{remark}
The equality concerns the infimum of the scalar curvature, which is approached along the ends and need not be attained at a finite point. The proposition establishes sharpness of the coefficient in Theorem~\ref{mthm:quantitative_Llarull_theorem}. It does not assert an equality classification.
\end{remark}
\section{Boundary theorem and rigidity}\label{sec:proof_boundary}

In this section, we prove Theorem~\ref{mthm:quantitative_Llarull_theorem_boundary}, Corollary~\ref{cor:geroch_conjecture_boundary} and Theorem~\ref{mthm:rigidity}.

\begin{proof}[Proof of Theorem~\ref{mthm:quantitative_Llarull_theorem_boundary}]
The proof follows the same strategy as in the proof of Theorem~\ref{mthm:quantitative_Llarull_theorem}, with the necessary adaptations to handle the boundary terms appearing in the spectral estimate.
Suppose, by contradiction, that 
\begin{equation}\label{eq:Llarull_assumption}
    \scalM \geq 0\quad \text{on}\quad M.
\end{equation}

Fix \(c=1\), and let \(\alpha_2=\alpha_2(c)\) be the corresponding constant in the spectral estimates above.

We begin with the case where $m$ is odd. 
Set $K_0:=\supp(d\Phi)$, $U:= \{ p \in M \colon d_p \Phi \neq 0 \}$ and $U_\delta := \{ p \in U \colon \mathbf{a}(p) < \delta \}$
for some constant $0 < \delta < 1$, where $\mathbf{a}(p)$ is the area contraction constant of $\Phi$ at $p$.
Note that $U_{\delta}\neq\emptyset$ by the same argument as in Proposition~\ref{pro:spectral_quantitative_Llarull_theorem}.
Since $K_0$ is compact and lies in the interior of $M$, the scalar-curvature assumption and continuity give a neighborhood $\mathcal{O}$ of $K_0$ with compact closure contained in the interior of $M$ and $\scalM\geq r_0>0$. Choose a compact smooth domain $K$ with $K_0\subset\operatorname{int}(K)\subset K\subset\mathcal{O}$, and a compact submanifold $\mathcal{W}$ containing $\partial M$ and $\overline{\mathcal{O}}$, with smooth boundary and interior containing $K$.
Let $\mathcal{U}$ be an open neighborhood of $K$ in $\mathcal{W}^{\circ}$ with $\overline{\mathcal{U}}\subset\mathcal{O}$. Choose $0<\scal_0<\min\{r_0,m(m-1)\}$, so that $\scalM\geq\scal_0$ on $\mathcal{U}$.
Let \(\psi\colon M\to[0,1]\) be a smooth cut-off function with \(\psi=0\) on \(K\) and \(\psi=1\) outside \(\overline{\mathcal{U}}\).
Choose $\varepsilon>0$ sufficiently small so that
\[
 \varepsilon\alpha_2\sup_{\overline{\mathcal{U}}\setminus U}|d\psi|
 < \frac{m\,\scal_0}{4(m-1)}.
\]
Set \(f:=\varepsilon\psi\). Then \(f\) is admissible, with \(f|_K=0\) and \(f|_{M\setminus\overline{\mathcal{U}}}=\varepsilon\).

Using the same constructions as in the proof of Theorem~\ref{mthm:quantitative_Llarull_theorem}, there exist a GL-pair $(\mathcal{E},\mathcal{F})$ over $M$ with support $K$ and a block unitary automorphism $\rho_M=\rho_M^+\oplus\rho_M^-$, with the two blocks agreeing and locally constant outside $K$, such that
\[
    \sfrel(M,\mathcal{E},\mathcal{F},\rho_M)=\deg(\Phi)\neq 0.
\]

Let $S$ be a relative Dirac bundle over $M$ associated to $(\mathcal{E},\mathcal{F})$ and let $\mathcal{D}$ be the corresponding Dirac operator on $S$ (cf.~\cite{CZ24}*{Example~2.5}).

We consider a family of Callias operators $\mathcal{B}_{f,-1}(t)=(1-t)\mathcal{B}_{f,-1} + t \rho_M^{-1} \mathcal{B}_{f,-1} \rho_M$ on $S$ for $t\in [0,1]$, where $\mathcal{B}_{f,-1}=\mathcal{D}+f\sigma$ is the Callias operator on $S$ subject to the sign $s=-1$. 
Let $\sflow(\mathcal{B}_{f,-1},\rho_M)$ denote the spectral flow of $\{\mathcal{B}_{f,-1}(t)\}_{t\in [0,1]}$.
By the same reasoning used in the proof of Theorem~\ref{mthm:quantitative_Llarull_theorem}, we see that 
\[
    \sflow(\mathcal{B}_{f,-1},\rho_M)= \sfrel(M,\mathcal{E},\mathcal{F},\rho_M) \neq 0.
\]
Thus there exist some $t\in [0,1]$ and a nonzero $u\in \ker(\mathcal{B}_{f,-1}(t))$. By Lemma~\ref{lem:kernel_regularization}, with the boundary assumptions $s=-1$, $f\geq0$, and $H\geq0$, we have $u\in H^1(M,S)$ and $|u|\in H^1(M)$, so the compactly supported integral estimates extend to $u$. From \eqref{eq:spectral_flow_connection_spectral_estimate} and \eqref{eq:odd_curvature_endomorphism}, we have
\begin{equation}\label{eq:noncompact_Llarull_theorem_with_boundary}
\begin{aligned}
0 \geq & \dfrac{m}{4c(m-1)} \int_M|d |u| |^2\,dV\\
  & + \frac{m}{m-1} \int_M
      \left( \frac{1}{4} \scalM |u|^2
      + \langle u, \mathscr{R}^{\mathcal{E} \oplus \mathcal{F}}(t) u \rangle \right)\,dV\\
  & + \int_M \langle u, \alpha_2 f^2 u + \alpha_2 c(d f) \sigma u \rangle\,dV\\
  & + \int_{\partial M}
      \underbrace{\Bigl(\alpha_2 f + \frac{m}{2} H\Bigr)}
      _{\geq 0,\ \text{by }\alpha_2>0,\ f,H\geq0}|u|^2\,dA,
\end{aligned}
\end{equation}
where \(c=1\) is the fixed choice above.

Because \(\Phi\) is area-decreasing, its area contraction constant satisfies \(\mathbf{a}(p)\le 1\) at every \(p\in M\).
Thus $\scalM(p) \ge \mathbf{a}(p)\cdot m(m-1)$ for all $p\in U$.
Recall that $\supp(d \psi)\subset \overline{\mathcal{U}}\setminus\operatorname{int}(K)\subset \overline{\mathcal{U}}\setminus U$. Observe that $\mathscr{R}^{\mathcal{E}\oplus\mathcal{F}}(t)=0$ on $ M\setminus U$.
Using \eqref{eq:curvature_term_area_decreasing} and \eqref{eq:noncompact_Llarull_theorem_with_boundary}, we deduce that
\begin{equation}
 \begin{aligned}
     0 \geq &  \dfrac{m}{4c(m-1)} \int_M|d |u| |^2\,dV\\
            & + \frac{m}{4(m-1)} \int_U
              \underbrace{\Bigl(\scalM-\mathbf{a}(p)\,m(m-1)\Bigr)}_{\geq 0}
              |u|^2\,dV\\
            & + \int_{M\setminus U}
              \Bigl(\frac{m\,\scalM}{4(m-1)}
              +\alpha_2 f^2-\alpha_2|d f|\Bigr)|u|^2\,dV.
 \end{aligned}
\end{equation}

Note that \(\scalM - \mathbf{a}(p) \cdot m(m-1) \ge (1-\delta)m(m-1) > 0\) on \(U_\delta \subset U\). Therefore,
\begin{equation}\label{eq:Llarull_compact_boundary_spectral_estimate}
 \begin{aligned}
     0 \geq  & \dfrac{m}{4c(m-1)} \|d |u| \|_{L^2(M)}^2\\
            & + \frac{1}{4}(1-\delta)m^2\|u\|^2_{L^2(U_{\delta})}\\
            & + \Bigl(\frac{m\,\scal_0}{4(m-1)}
              -\varepsilon\alpha_2\sup_{\overline{\mathcal{U}}\setminus U}|d\psi|\Bigr)
              \|u\|^2_{L^2(\overline{\mathcal{U}}\setminus U)}\\
            & + \underbrace{\Bigl(\frac{m}{4(m-1)}
              \inf_{M\setminus\overline{\mathcal{U}}}\scalM
              +\varepsilon^2\alpha_2\Bigr)}
              _{>0,\ \text{by }\eqref{eq:Llarull_assumption}}\,
              \|u\|_{L^2(M\setminus\overline{\mathcal{U}})}^2.
 \end{aligned}
\end{equation}
The preceding choice of $\varepsilon$ makes the coefficient of $\|u\|^2_{L^2(\overline{\mathcal U}\setminus U)}$ strictly positive.
Note that the four terms on the right-hand side of \eqref{eq:Llarull_compact_boundary_spectral_estimate} are nonnegative. Hence we must have $\|d |u|\|_{L^2(M)} = 0$ and $u=0$ on $U_{\delta} \cup (M\setminus U)$.
Elliptic regularity for the Callias boundary problem gives the regularity needed to interpret this conclusion pointwise. The vanishing of the differential implies that $|u|$ is constant on $M$. Together with $u=0$ on the nonempty open set $U_{\delta}$ we conclude $|u| \equiv 0$ and thus $u \equiv 0$ on $M$, a contradiction.
We have proved this theorem for $m$ odd.

When $m$ is even, use the even-dimensional GL-pair construction from Case~1 in the proof of Proposition~\ref{pro:spectral_quantitative_Llarull_theorem}, with the same compact support $K$. The doubled map has nonzero degree, so the resulting pair satisfies $\indrel(M,\mathcal{E},\mathcal{F})\neq0$. The Callias index splitting and the correspondence theorem of~\cite{CZ24}*{Theorem~3.6, Corollary~3.9}, with the same choice $s=-1$, therefore give a nonzero index for $\mathcal{B}_{f,-1}$. Hence one of its two graded components has a nonzero kernel. By Lemma~\ref{lem:kernel_regularization}, this kernel section belongs to $H^1$ and the compactly supported integral estimates extend to it. Applying \eqref{eq:even_connection_spectral_estimate} gives the boundary spectral inequality, whose boundary term is nonnegative after moving the left-hand side to the right because $f\geq0$ and $H\geq0$. The curvature estimate and the choice of $\varepsilon$ then give exactly \eqref{eq:Llarull_compact_boundary_spectral_estimate}. The same nonnegativity argument forces $|u|$ to be constant and to vanish on $U_\delta$, as in the odd-dimensional case. This proves the theorem in even dimensions as well.
Hence the proof of Theorem~\ref{mthm:quantitative_Llarull_theorem_boundary} is complete. 
\end{proof}

\begin{proof}[Proof of Corollary~\ref{cor:geroch_conjecture_boundary}]
Assume, for a contradiction, that $\mathcal M=W\#M$ has a complete metric $g_{\mathcal M}$ with $\scal_{g_{\mathcal M}}>0$ and mean-convex boundary. Let $B_{\mathrm W}\subset W^\circ$ be the ball removed from the $W$ side to form the connected sum, and let $D\subset M$ be the ball removed from the $M$ side. Their boundaries are identified in the connected sum. Choose nested coordinate balls $B_{\mathrm W}\subset B_0\subset B_1\subset W^\circ$ in an oriented chart $\varphi:B_1\to B_R(0)\subset\mathbb R^m$, with $\varphi(\overline{B_0})\subset B_{r_0}(0)$ for some $r_0<R$. Choose a smooth function $\beta:[0,R]\to[0,1]$ such that $\beta=0$ on $[0,r_0]$ and $\beta=1$ on a neighborhood of $R$. Define
\[
p(x)=\begin{cases}
\varphi^{-1}\!\bigl(\beta(|\varphi(x)|)\varphi(x)\bigr),&x\in B_1,\\
x,&x\in W\setminus B_1.
\end{cases}
\]
Since $\beta=1$ near $\partial B_R$, this defines a smooth map equal to the identity near $\partial B_1$, and it is constant on a neighborhood of $B_0$. The homotopy
\[
p_t(x)=\begin{cases}
\varphi^{-1}\!\bigl(((1-t)+t\beta(|\varphi(x)|))\varphi(x)\bigr),&x\in B_1,\\
x,&x\in W\setminus B_1,
\end{cases}\qquad 0\leq t\leq1,
\]
is fixed on $W\setminus B_1$, starts at the identity, and ends at $p$. Hence $p$ is homotopic to the identity relative to $W\setminus B_1$ and has degree one. Fix a reference metric $h$ on $W$ and put
\[
W_0:=W\setminus\operatorname{int}(B_0),\qquad
\kappa:=\min_{W_0}\scal_{g_{\mathcal M}}>0,\qquad
C:=\sup_{W_0}\|\Lambda^2dp\|_{g_{\mathcal M},h}<\infty,
\]
where $g_{\mathcal M}$ is restricted to the compact $W$-side of the connected sum. Choose $r>0$ with $r^2\leq\kappa/[m(m-1)]$.

By area-enlargeability, choose a finite cover $\widehat W\to W$ and a smooth map $q:\widehat W\to\sphm$, locally constant near $\partial\widehat W$, of nonzero degree and with area contraction at most $\epsilon$ for the lifted metric $h$, where $0<\epsilon\leq r^2/\max\{1,C\}$. If the cover is disconnected, retain a component on which the degree is nonzero. The homotopy from $p$ to the identity, fixed off $B_1$, lifts from the identity and gives a map $\widehat p$ which is the identity off the lifted copies of $B_1$ and has degree one. Since $\widehat p$ is constant on a neighborhood of every lift of $B_{\mathrm W}$, the map $q\circ\widehat p$ is constant there. Its area contraction is at most $C\epsilon$ with respect to the lifted metric $g_{\mathcal M}$ on the $W$-side.

The finite cover of $W$ extends over $W\#M$ by attaching one copy of $M\setminus\operatorname{int}(D)$ for each lift of $B_{\mathrm W}$. Extend $q\circ\widehat p$ constantly over these copies, using its constant value near each gluing sphere. The resulting smooth map has nonzero degree: a regular value outside the finitely many constant values has the same preimage count as on $\widehat W$. It is locally constant near the physical boundary and at infinity, and its differential is supported in the compact lift of $W_0$. On the rescaled metric $r^2g_{\mathcal M}$, its area contraction is at most $r^{-2}C\epsilon\leq1$, while
\[
\scal_{r^2g_{\mathcal M}}=r^{-2}\scal_{g_{\mathcal M}}\geq r^{-2}\kappa\geq m(m-1)
\]
on the support of the differential. Constant rescaling preserves completeness and the sign of the boundary mean curvature. The boundary theorem therefore gives $\inf_{\widehat{\mathcal M}}\scal_{r^2g_{\mathcal M}}<0$, contradicting $\scal_{g_{\mathcal M}}>0$ everywhere.
\end{proof}

Finally, we prove Theorem~\ref{mthm:rigidity}.
\begin{proof}[Proof of Theorem~\ref{mthm:rigidity}]
Assume that the Ricci curvature \(\Ric_{g_{\mathcal{M}}}\) of $g_{\mathcal{M}}$ is not identically zero, then by a deformation result due to Cruz--Santos~\cite{CS25+}*{Theorem~1.1} the manifold \(\mathcal{M}\) admits a complete metric with positive scalar curvature and mean-convex boundary. This is impossible by Corollary~\ref{cor:geroch_conjecture_boundary}. Consequently, \(\Ric_{g_{\mathcal{M}}} \equiv 0\).

By Ichida~\cite{Ichida81}, $\partial\mathcal{M}$ has at most two connected components. If it had two components, the splitting theorem for manifolds with boundary of Ichida~\cite{Ichida81} (or Kasue~\cite{Kas83}*{Theorem~B}, Croke--Kleiner~\cite{CK92}*{Theorem~1}) would make $\mathcal{M}$ a compact product, contradicting the noncompactness of $M$. Thus $\partial\mathcal{M}$ is connected. Since $M$ is noncompact, Kasue's boundary splitting theorem~\cite{Kas83}*{Theorem~C} (or Croke--Kleiner~\cite{CK92}*{Theorem~2}) applies and yields the asserted product decomposition.
\end{proof}


\section*{Acknowledgments}
The author is deeply grateful to Professor Weiping Zhang for his insightful discussions and helpful suggestions.
The author thanks Professor Guangxiang Su for his useful discussions and valuable comments. The author also thanks Professor Zhenlei Zhang and Professor Bo Liu for their continuous encouragement and support.
This work is partially supported by
the National Natural Science Foundation of China Grant No. 12501064, 
the China Postdoctoral Science Foundation (Grant No. 2025M773075, Postdoctoral Fellowship Program Grant No. GZC20252016, and Tianjin Joint Support Program Grant No. 2025T002TJ), and the Nankai Zhide Foundation.


\end{document}